\documentclass[10pt]{article}
\usepackage[a4paper]{geometry}
\usepackage{amsmath, amssymb, amsthm
}
\usepackage{appendix}
\usepackage{multicol} 
\usepackage{makeidx}
\usepackage{cite}
\usepackage{captdef}
\makeindex
\usepackage{paralist}

\usepackage{psfrag}
\usepackage{graphicx}
\usepackage{graphics}
\usepackage{color}

\usepackage[showonlyrefs]{mathtools}
\mathtoolsset{showonlyrefs=true}

\usepackage{subcaption}

\newcommand{\R}{\mathbb{R}}

\newtheorem{theo}{Theorem}[section]
\newtheorem{coro}[theo]{Corollary}
\newtheorem{lemma}[theo]{Lemma}
\newtheorem{prop}[theo]{Proposition}
\theoremstyle{definition}

\newtheorem{remark}[theo]{Remark}

\newcommand{\mF}{{\mathfrak F}}

\newcommand{\cI}{{\mathcal I}}

\newcommand{\cN}{{\mathcal N}}

\newcommand{\supp}{{\rm supp}}

\newcommand{\eps}{\varepsilon}

\newcommand{\N}{{\mathbb{N}}}

\renewcommand{\epsilon}{\varepsilon}

\newcommand{\Sn}{{\mathbb S}^{N-1}}

\numberwithin{equation}{section}

\title{Least energy nodal solutions of Hamiltonian elliptic systems with Neumann boundary conditions}
\author{Alberto Salda\~{n}a\footnote{
Institut f\"ur Analysis, Karlsruhe Institute for Technology, Englerstra\ss e 2, 76131, Karlsruhe, Germany, \newline alberto.saldana@partner.kit.edu
}\ \ \& Hugo Tavares\footnote{CAMGSD, Departamento de Matem\'atica, Instituto Superior T\'{e}cnico, Universidade de Lisboa, Av. Rovisco Pais, 1049-001 Lisboa, Portugal; htavares@tecnico.ulisboa.pt} \footnote{Departamento de Matem\'atica, Faculdade de Ci\^encias da Universidade de Lisboa, Edif\'icio C6, Piso 1, Campo Grande 1749-016  Lisboa, Portugal; hrtavares@ciencias.ulisboa.pt}}

\date{\today}

\begin{document}
\maketitle

\begin{abstract}
We study existence, regularity, and qualitative properties of solutions to the
system
\[
 -\Delta u = |v|^{q-1} v\quad \text{ in }\Omega,\qquad -\Delta v = |u|^{p-1} u\quad \text{ in }\Omega,\qquad \partial_\nu u=\partial_\nu v=0\quad \text{ on }\partial\Omega,
\]
with $\Omega\subset \R^N$ bounded; in this setting, all nontrivial solutions are sign changing. Our proofs use a variational formulation in dual spaces, considering sublinear $pq< 1$ and superlinear $pq>1$ problems in the subcritical regime. In balls and annuli we show that least energy solutions (l.e.s.) are foliated Schwarz symmetric and, due to a symmetry-breaking phenomenon, l.e.s. are \emph{not} radial functions; a key element in the proof is a new $L^t$-norm-preserving transformation, which combines a suitable flipping with a decreasing rearrangement. This combination allows us to treat annular domains, sign-changing functions, and Neumann problems, which are non-standard settings to use rearrangements and symmetrizations. In particular, we show that our transformation diminishes the (dual) energy and, as a consequence, \emph{radial} l.e.s. are strictly monotone.  We also study unique continuation properties and simplicity of zeros.  Our theorems also apply to the scalar associated model, where our approach provides new results as well as alternative proofs of known facts.

\medbreak


\end{abstract}

%\tableofcontents

\section{Introduction}

Let $N\geq 1$, $\Omega\subset\R^N$ be a smooth bounded domain, and consider the following Hamiltonian elliptic system with Neumann boundary conditions 
\begin{align}\label{NHS}
 -\Delta u = |v|^{q-1} v\quad \text{ in }\Omega,\qquad -\Delta v = |u|^{p-1} u\quad \text{ in }\Omega,\qquad \partial_\nu u=\partial_\nu v=0\quad \text{ on }\partial\Omega,
\end{align}
where $\nu$ is the outer normal vector on $\partial \Omega$ and we consider $p,q>0$ in the \emph{sublinear} ($pq<1$) or the \emph{superlinear} ($pq>1$) cases satisfying a \emph{subcritical} condition, that is,
\begin{align}\label{sc:intro}
p,q>0,\qquad pq\neq 1,\qquad \text{ and }\qquad \frac{1}{p+1}+\frac{1}{q+1}>\frac{N-2}{N}.  
\end{align}
 In fact, in this setting the more general notion of \emph{linearity} is $pq=1$ or, equivalently, $1/(p+1)+1/(q+1)=1$ \cite{ClementvanderVorst}. On the other hand, the last inequality in \eqref{sc:intro}
means that the exponents $(p,q)$ are 
below the \emph{critical hyperbola} (\emph{i.e.}, $(p,q)$ is subcritical) \cite{ClementdeFigueiredoMitidieri, PeletiervanderVorst}, and this condition is trivially satisfied if $N=1,2$ or if $pq<1$.

Systems with a Hamiltonian structure such as \eqref{NHS} have been extensively studied in the past 25 years and many results are known regarding existence, multiplicity, concentration phenomena, positivity, symmetry, Liouville theorems, \emph{etc}.  We refer to the surveys \cite{deFigueiredo, Ruf, BMT14} for an overview of the topic and to \cite{BMRT15, new1, new2, new3} for more recent results.  Most of these papers use Dirichlet boundary conditions and, up to our knowledge, the few papers addressing Neumann problems are
\cite{Zeng, AvilaYang, RamosYang, PistoiaRamosNeumann}, where existence of positive solutions and concentration phenomena are studied, and \cite{BonheureSerraTilli}, which centers on existence of positive radial solutions. However, these papers focus on a \emph{different} operator of the form $Lw=-\Delta w+V(x) w$, with $V$ positive.  In comparison with \eqref{NHS}, the shape of solutions changes drastically; for instance, the operator $L$ with Neumann b.c. induces a norm, and this allows the existence of positive solutions, while all nontrivial solutions of \eqref{NHS} are \emph{sign-changing}. Indeed, if $(u,v)$ is a classical solution of \eqref{NHS}, then by the Neumann b.c. and the divergence theorem, 
\begin{align}\label{comp}
\int_\Omega |u|^{p-1}u=\int_\Omega |v|^{q-1}v=0.
\end{align}
Since $u\equiv 0 $ if and only if $v\equiv 0$, \eqref{comp} is only satisfied if $(u,v)$ is trivial or if both components are sign-changing. 
Condition \eqref{comp} is called a \emph{compatibility condition}.  As far as we know, our paper is the first to study problem \eqref{NHS}.

We remark that if $p=q>0$ and $(u,v)$ is a classical solution of \eqref{NHS}, then $u\equiv v$ in $\Omega$ and \eqref{NHS} reduces to the scalar equation $-\Delta u = |u|^{p-1}$ in $\Omega$ with $\partial_\nu u=0$ on $\partial\Omega$, see Lemma \ref{lemma:p=q}. Therefore, all our results cover the single equations case.

\medskip

Condition \eqref{sc:intro}  together with Sobolev embeddings and the Rellich--Kondrachov theorem implies that
\begin{align}\label{embed}
W^{2,\frac{p+1}{p}}(\Omega)\hookrightarrow L^{q+1}(\Omega)\quad  \text{ and }\quad  W^{2,\frac{q+1}{q}}(\Omega)\hookrightarrow L^{p+1}(\Omega)\quad \text{ compactly}.
\end{align}
A strong solution of \eqref{NHS} is defined as a pair $(u,v)\in W^{2,\frac{q+1}{q}}(\Omega)\times W^{2,\frac{p+1}{p}}(\Omega)$ satisfying the equations a.e. in $\Omega$, and the boundary conditions in the trace sense.
%In fact, all strong solutions are classical solutions, see Proposition \ref{p:reg} below.
Problem \eqref{NHS} has a variational structure, and \eqref{NHS} are the Euler-Lagrange equations of the energy functional
\begin{equation}\label{eq:functional_I}
(u,v)\mapsto I(u,v)= \int_\Omega \nabla u\cdot \nabla v-\frac{|u|^{p+1}}{p+1}-\frac{|v|^{q+1}}{q+1}\, dx.
\end{equation}
We define a \emph{least energy (nodal) solution} as a nontrivial strong solution of \eqref{NHS} achieving the level
\begin{align}\label{c:level}
c:=\inf\left\{ I(u,v):\ (u,v)\not\equiv (0,0),\ (u,v) \text{ is a strong solution of \eqref{NHS}}\right\}. 
\end{align}
In view of \eqref{sc:intro} and \eqref{embed}, the functional $I$ is well defined at strong solutions. Our main result is concerned with existence, regularity, and qualitative properties of least energy solutions.

\begin{theo}\label{thm:maintheorem1}
Let $N\geq 1$, $\Omega\subset \R^N$ be a smooth bounded domain, and let $p$ and $q$ satisfy \eqref{sc:intro}.  The set of least energy solutions is nonempty. If $(u,v)$ is a least energy solution, then $(u,v)\in C^{2,\varepsilon}(\overline{\Omega})\times C^{2,\varepsilon}(\overline{\Omega})$ is a classical solution of \eqref{NHS} and the following holds.
\begin{itemize}
\item[(i)] (Monotonicity in 1D) If $N=1$ and $\Omega=(-1,1)$, then $u'v'>0$ in $\Omega$; in particular, $u$ and $v$ are both strictly monotone increasing or both strictly monotone decreasing in $\Omega$.  
\item[(ii)] (Symmetry \& symmetry breaking) If $N\geq 2$ and $\Omega=B_1(0)$ or $\Omega=B_1(0)\backslash B_\delta(0)$ for some $\delta\in(0,1)$, then there is $e\in \partial B_1(0)$ such that $u$ and $v$ are foliated Schwarz symmetric with respect to $e$ and the functions $u$ and $v$ are \emph{not} radially symmetric.
\item[(iii)] (Unique continuation property) If $pq<1$, then the zero sets of $u$ and $v$ have zero Lebesgue measure, \emph{i.e.}, $|\{x\in \Omega:\ u(x)=0\}|=|\{x\in \Omega:\ v(x)=0\}|=0.$
\end{itemize} 
\end{theo}
Results for the scalar equation follow immediately from Theorem \ref{thm:maintheorem1}, see Corollary~\ref{thm:singleeq} below.  Our approach to show Theorem \ref{thm:maintheorem1} is based on a variant of the \emph{dual method} \cite{ClementvanderVorst,AlvesSoares}.
Later in this introduction we motivate the use of this approach and also the relationship between our results and previously known results for the single-equation problem. To describe the dual framework, we introduce some notation used throughout the whole paper. Let $p$ and $q$ satisfy \eqref{sc:intro} and, for $s>1$, let
\begin{align}
X^s=\Big\{f\in L^{s}(\Omega): \ \int_\Omega f =0\Big\},\qquad
\alpha:=\frac{p+1}{p},\qquad
\beta:=\frac{q+1}{q},\qquad \text{and}\qquad
X:=X^\alpha\times X^\beta,
\end{align}
endowed with the norm $\|(f,g)\|_X=\|f\|_\alpha+\|g\|_\beta$. Let $K$ denote the inverse (Neumann) Laplace operator with zero average, that is, if $h\in X^s(\Omega)$, then $u := Kh\in W^{2,s}(\Omega)$ is the unique strong solution of $-\Delta u = h$ in $\Omega$ satisfying $\partial_\nu u=0$ on $\partial \Omega$ and $\int_\Omega u = 0$, see Lemma \ref{l:reg} below.  In this setting, the (dual) energy functional $\phi:X\to \R$ is given by
\begin{align}\label{phi:def}
 \phi(f,g):=\int_\Omega \frac{|f|^{\alpha}}{\alpha}+\frac{|g|^{\beta}}{\beta}-g\, K f\ dx,\qquad (f,g)\in X.
\end{align}
Since \eqref{comp} holds for any nontrivial solution, we require a suitable translation of $K$. For $t>0$, let $K_t:X^\frac{t+1}{t}\to W^{2,\frac{t+1}{t}}(\Omega)$ be given by 
\begin{align}\label{Ks:def}
K_t h:=Kh +c_t(h)\qquad \text{ for some }\quad c_t(h)\in\R\quad \text{ such that }\int_\Omega |K_t h|^{t-1}K_t h=0.  
\end{align}
Then, a critical point $(f,g)$ of $\phi$ solves the dual system $K_{q} f=|g|^{\frac{1}{q}-1}g$ and $K_{p} g= |f|^{\frac{1}{p}-1}f$ in $\Omega$, see Lemma \ref{l:ss}.

In the sublinear case ($pq<1$), the operator $\phi$ achieves its global minimum in $X$ (see Lemma \ref{minimum}); whereas, in the superlinear case ($pq>1$), $\phi$ is unbounded from below, but it can be minimized (see Lemma \ref{lemma:B}) in the Nehari-type set
\begin{align*}
\mathcal{N}&:=\{(f,g)\in X\backslash \{(0,0)\} : \ \phi'(f,g)(\gamma_1f,\gamma_2g)=0\}
\quad \text{ with }\quad \gamma_1:=\frac{\beta}{\alpha+\beta},\quad \gamma_2:=\frac{\alpha}{\alpha+\beta}.
\end{align*}
Therefore, we often refer to a minimizer $(f,g)\in X$ satisfying
\begin{align}\label{lee}
\phi(f,g)=\inf_{X}\phi\quad \text{ if } pq<1 \qquad \text{ or }\qquad \phi(f,g)=\inf_{{\cal N}}\phi\quad \text{ if } pq>1.
\end{align}
In particular, if $(f,g)$ satisfies \eqref{lee}, then $(u,v):=(K_p g, K_q f)$ is a \emph{least energy solution}, that is, $\phi(f,g)=I(u,v)=c$, with $c$ as in \eqref{c:level} and solves \eqref{NHS}, see Lemmas \ref{l:same}, \ref{minimum}, and \ref{lemma:B} below. 

\bigskip

We now describe in more detail the different techniques involved in the proof of Theorem~\ref{thm:maintheorem1}.  The existence is obtained 
using the subcriticality assumption \eqref{sc:intro} and the compactness of the operator $K$, while the regularity follows from a 
bootstrap argument (Proposition \ref{p:reg}). 
The unique continuation property is shown by extending the techniques from \cite[Section 3]{PW15} to the setting of Hamiltonian elliptic systems and to the dual-method framework (later in this introduction we compare in more detail the results and strategies from \cite{PW15} with ours). Moreover, the fact that least energy solutions are foliated Schwarz symmetric is due to a characterization of sets of functions with this symmetry in terms of invariance under polarizations, see Lemma \ref{l:char} below and \cite{BartschWethWillem, PW15, BMRT15, B03, SW12} for similar results in other settings.

One of our main contributions, from the methodological point of view, is the proof of the strict monotonicity of solutions (for $N=1$) and of the symmetry-breaking phenomenon (\emph{i.e.}, that least energy solutions are \emph{not} radial).  These results require first a deep understanding of the radial setting: consider $\Omega\subset \R^N$ to be either
\begin{equation}\label{Omega:eqs}
\begin{aligned}
\Omega&=B_1(0)\qquad \qquad \text{ and fix $\delta:=0$}\qquad \quad \text{(for any $N\geq 1$)}\qquad\text{ or }\\
\Omega&=B_1(0)\backslash B_\delta(0)\quad \text{ for some }\delta\in(0,1)\quad \text{(for any $N\geq 2$)}.
\end{aligned}
\end{equation}
If $Y$ is a set of functions, we use $Y_{rad}$ to denote the subset of radial functions in $Y$ and we call $(u,v):=(K_{p} g, K_{q} f)$ a \emph{least energy radial solution} if $(f,g)$ 
satisfies 
\begin{align}\label{leer}
\phi(f,g)=\inf_{X_{rad}}\phi\quad \text{ if } pq<1 \qquad \text{ or }\qquad \phi(f,g)=\inf_{{\cal N}_{rad}}\phi\quad \text{ if } pq>1.
\end{align}
Our next result shows that these radial solutions are remarkably rigid. We denote the radial derivative of $w$ by $w_r$ and use a slight abuse of notation setting $w(|x|)=w(x)$.

\begin{theo} \label{thm:maintheorem3}
Let $p,q$ satisfy \eqref{sc:intro} and $\Omega,$ $\delta$ as in \eqref{Omega:eqs}. The set of least energy radial solutions is nonempty and, if $(u,v)$ is a least energy radial solution, then $(u,v)\in C^{2,\varepsilon}(\overline{\Omega})\times C^{2,\varepsilon}(\overline{\Omega})$ is a classical radial solution of \eqref{NHS} and $u_r v_r>0$ in $(\delta,1)$; in particular, $u$ and $v$ are both strictly monotone increasing or both strictly monotone decreasing in the radial variable.
\end{theo}
The existence and regularity follows similarly as in Theorem \ref{thm:maintheorem1}.  The strict monotonicity relies on the following new transformation denoted with a superindex~$\divideontimes$:
let $\Omega$, $\delta$ as in \eqref{Omega:eqs} and
\begin{align*}
&{\cal I}:L_{rad}^\infty(\Omega)\to C_{rad}(\overline{\Omega}),\qquad {\cal I}h(x):=\int_{\{\delta<|y|<|x|\}} h(y)\ dy
=N\omega_N\int_\delta^{|x|}h(\rho)\rho^{N-1}\ d\rho,\\
&\mF: C_{rad}(\overline{\Omega})\to L_{rad}^\infty(\Omega),\qquad \mF h:=(\chi_{\{{\cal I}h>0\}}-\chi_{\{{\cal I}h\leq 0\}})\, h.
\end{align*}
Then, for $h\in C_{rad}(\overline{\Omega})$, the $\divideontimes$-transformation of $h$ is given by 
\begin{align*}
h^{\divideontimes}\in L^\infty_{rad}(\Omega),\qquad h^{\divideontimes}(x) := ({\mathfrak F} h)^\#(\omega_N |x|^N-\omega_N \delta^N).
\end{align*}
where $\#$ is the one-dimensional decreasing rearrangement (see Subsection \ref{subsec:dr}) and $\omega_N= |B_1|$ is the volume of the unitary ball in $\R^N$.  The function $\mF h$ can be seen as a suitable flipping of $h$ on the set $\{\cI h\leq 0\}$, and this step is very important to construct monotone decreasing solutions.  In Remark \ref{rem:in} we motivate further the definition of $h^\divideontimes$ and in Figure \ref{fig} in Section \ref{sec:transformation} we illustrate the construction of $h^\divideontimes$ in a ball and compare it with the Schwarz symmetrization $h^*(x)=h^\#(\omega_N|x|^N)$. 

Observe that annuli, sign-changing functions, and Neumann boundary data are non-standard conditions to work with rearrangements; in fact, many results relying on symmetrizations fail in these settings; for example, the standard Polya-Szeg\H{o} inequality only holds for nonnegative functions with zero boundary data.  Our approach is able to cover these cases mainly because of two reasons: the \emph{radiality} assumption and the fact that we use \emph{Lebesgue spaces} within a dual framework, which gives us more flexibility in the construction of our transformation; in this sense, we are transforming the \emph{dual variables} $(f,g)$ to obtain, together with variational techniques, monotonicity information of solutions $(u,v):=(K_p g, K_q f)$.

In general, $h^\divideontimes$ is \emph{not} a (level-set) rearrangement of $h$, since the maximum value of $h$ may vary due to the flipping ${\mathfrak F}h$; however, $L^t$-norms and (zero) averages are preserved (see Proposition \ref{prop:prop}) and the transformed functions have less energy, as stated in the following result.
\begin{theo}\label{thm:mono:intro} Let $p,q>0$, $\Omega$ as in \eqref{Omega:eqs}, and let $f,g:\overline{\Omega}\to \R$ be continuous and radially symmetric functions with $\int_\Omega f = \int_\Omega g = 0$. Then $(f^\divideontimes,g^\divideontimes)\in X$ and
\begin{align}\label{goal:mono:intro}
\phi(f^\divideontimes,g^\divideontimes)\leq \phi(f,g).
\end{align}
Furthermore, if $f$ and $g$ are nontrivial and $\phi(f^\divideontimes,g^\divideontimes)=\phi(f,g)$, then $f$ and $g$ are monotone in the radial variable and if $(u,v):=(K_{p} g,K_{q} f)$, then $u$ and $v$ are radially symmetric and strictly monotone in the radial variable.
\end{theo}
The proof exploits the one-dimensionality of the problem and uses elementary rearrangement techniques.  
Note that, if $\Omega$ is a ball, then $h^\divideontimes$ is actually the Schwarz symmetrization of $\mathfrak{F}h$. The flipping $\mF$, however, is necessary, since it can be shown that \eqref{goal:mono:intro} does \emph{not} hold in general using merely the Schwarz symmetrization (even in the scalar case $p=q$), see Remark \ref{ce:rem} below. 

Theorem \ref{thm:mono:intro} is the main tool to show the monotonicity claims in Theorem \ref{thm:maintheorem1} and Theorem \ref{thm:maintheorem3}. These results are of independent interest and are new even in the single-equation case \eqref{eq:1eq_problem}. Furthermore, Theorem \ref{thm:mono:intro} is also the starting point for the proof of the symmetry-breaking phenomenon, which, in the dual setting, reads as follows.
\begin{align}\label{claim:nr:thm}
 \text{If $(f,g)$ satisfies \eqref{lee}, then $f$ and $g$ are not radial}.
\end{align}
The proof of this claim is done by contradiction: if $(f,g)$ is radially symmetric and $(u,v):=(K_{p} g,K_{q} f)$, then $(u_{x_1},v_{x_1})$ is a strong (Dirichlet) solution of
\begin{align*}
 -\Delta u_{x_1} &= q |v|^{q-1} v_{x_1}=:\bar g,\ \ -\Delta v_{x_1} = p |u|^{p-1} u_{x_1}=:\bar f \ \ \text{ in }\Omega\qquad\text{with}\quad u_{x_1}=v_{x_1}=0\quad \text{ on }\partial\Omega.
\end{align*}
Therefore, from the minimality of $(f,g)$ (approximating $\phi''$ in a suitable sense) we infer that
\begin{align}\label{int:intro}
\int_{\Omega(e_1)} p|u|^{p-1}u_{x_1}(u_{x_1}-K \bar g )+q|v|^{q-1} v_{x_1}(v_{x_1}-K \bar f)\, dx\geq 0,\quad \Omega(e_1):=\Omega\cap\{x_1>0\}.\qquad 
\end{align}
Here, Theorem \ref{thm:maintheorem3} is very important to control the (possibly singular) terms $|u|^{p-1}$ and $|v|^{p-1}$, since the strict monotonicity implies that the nodal sets of $u$ and $v$ are merely two inner spheres. In the end, a contradiction is obtained by showing\textemdash using the radiality assumption, maximum principles, and Hopf's boundary point Lemma\textemdash that the Neumann solution $(K\bar g,K\bar f)$ dominates the Dirichlet solution $(u_{x_1},v_{x_1})$ in $\Omega(e_1)$, and this would imply that the integral in \eqref{int:intro} is strictly negative.  Observe that
the symmetry-breaking statement in Theorem \ref{thm:maintheorem1} follows directly from \eqref{claim:nr:thm}. 
For other symmetry-breaking results
for single equations
we refer to \cite{AftalionPacella} for Dirichlet boundary conditions, to \cite{GiraoWeth,PW15} for Neumann boundary conditions, and to \cite{BMT14} for a (perturbative) symmetry-breaking result for Dirichlet Hamiltonian systems. See also the survey \cite{W10} and the references therein.
\medskip

We now focus on the particular case $p=q$, where \eqref{sc:intro} reduces to
\begin{equation}\label{eq:p_single}
p>0,\qquad p\neq 1,\qquad (N-2)p<(N+2).
\end{equation}
In this situation, we show in Lemma \ref{lemma:p=q} below that any classical solution $(u,v)$ of \eqref{NHS} satisfies $u\equiv v$, and problem \eqref{NHS} is equivalent to
\begin{align}\label{eq:1eq_problem}
-\Delta u= |u|^{p-1}u \quad \text{ in }\Omega,\qquad  \partial_\nu u=0\quad \text{ on }\partial\Omega,
\end{align}
whose solutions are critical points of 
\begin{align}\label{E:f}
E(u):=\frac{I(u,u)}{2}=\int_\Omega \frac{1}{2}|\nabla u|^2 - \frac{1}{p}|u|^{p}\ dx 
\end{align}
Then, as a particular case of Theorems \ref{thm:maintheorem1} and \ref{thm:maintheorem3} we have the following.
\begin{coro}\label{thm:singleeq} Let $N\geq 1$, $\Omega\subset \R^N$ be a smooth bounded domain, and let $p$ satisfy \eqref{eq:p_single}.
The set of least energy solutions of \eqref{eq:1eq_problem} is nonempty and it is contained in $C^{2,\varepsilon}(\overline{\Omega})$. 
\begin{itemize}
\item[(i)] (Unique continuation) If $p<1$, then the zero set of every least energy solution has zero Lebesgue measure.
\item[(ii)] (Monotonicity, symmetry, and symmetry breaking)  If $N=1$ and $\Omega=(-1,1)$, then every least energy solution is strictly monotone in $\Omega$.  If $N\geq 2$ and 
$\Omega$ is either a ball or an annulus as in \eqref{Omega:eqs}, then every least energy solution is foliated Schwarz symmetric and it is \emph{not} radially symmetric.
\item[(iii)] (Radial solutions) Let $\Omega$ be a ball or an annulus as in \eqref{Omega:eqs}, then the set of least energy \emph{radial} solutions is nonempty. If $u$ is a least energy radial solution, then $u\in C^{2,\varepsilon}(\overline{\Omega})$ is a classical radial solution of \eqref{eq:p_single} and $u$ is strictly monotone in the radial variable.
\end{itemize} 
\end{coro}
Up to our knowledge, the monotonicity of least energy radial solutions (part $(iii)$ of Corollary \ref{thm:singleeq}) is new. Part (ii) is also new in the subcritical superlinear regime as well as the symmetry breaking result for annuli in the sublinear case. In this sense, this corollary complements the results in \cite{PW15}, where unique continuation, symmetry, and symmetry breaking for least energy solutions of \eqref{eq:1eq_problem} in the case $0< p<1$ are studied (the case $p=0$ is also considered in \cite{PW15}, interpreting \eqref{eq:1eq_problem} as $-\Delta u=\text{sign}(u)$). Note that the unique continuation property when $p>1$ is classical, see for instance \cite{JK85, FG92,GL93}.  Next, we comment on the similarities and differences between our approach and that of \cite{PW15}; in particular, we explain the difficulties when passing from the single-equation \eqref{eq:1eq_problem} to system \eqref{sc:intro} and why the dual method is convenient in this last case.

To show existence of solutions (in the sublinear case $0<p<1$), the authors in \cite{PW15} minimize the functional \eqref{E:f} on the set
\begin{align}\label{N:set}
N_p=\{u\in H^1(\Omega):\ \int_\Omega |u|^{p-1}u=0\}, 
\end{align}
proving that the least energy level is achieved (note that $N_p$ is \emph{not} a $C^1$-manifold because $p<1$). Several difficulties arise when trying to use such a direct method for a Hamiltonian system. For instance, in dimension $N\geq 3$, the (direct) functional $I$, given in \eqref{eq:functional_I}, is not well defined in $H:=H^1(\Omega)\times H^1(\Omega)$ under \eqref{sc:intro}. In fact, even working on a range of $(p,q)$ where $I$ is well defined in $H$
(or using \eqref{sc:intro} with $p,q>1$ and a truncation method as in \cite{RamosTavares}), the functional $I$ is \emph{strongly indefinite}, in the sense that the principal part $\int_\Omega \nabla u\cdot \nabla v$ does not have a sign, and it is actually positive in an infinite dimensional subspace, and negative in another.  To control this difficulty, a direct approach is based, for example, on abstract linking theorems or on special Nehari-type sets.
Different alternatives are available in the literature to study Hamiltonian systems variationally (see, \emph{e.g.}, the survey \cite{BMT14}), each one with its own advantages and disadvantages. For example, a common strategy is to reduce the system to a single higher-order problem; however, it is not clear how to study Neumann b.c. in this setting and using higher-order Sobolev spaces brings additional complications (for the use of rearrangements, for example). 

Dual methods, on the other hand, offer a flexible and elegant alternative. This approach entails the challenge of controlling the effects of the nonlocal operator $K$ in the functional~$\phi$; but it compensates this difficulty with many advantages, for example, the compatibility conditions \eqref{comp} translate to $\int_\Omega f=\int_\Omega g=0$ in the dual formulation; in particular, this allows in the sublinear case to minimize and \emph{differentiate} the functional $\phi$ in the Banach space $X$ (recall that in \cite{PW15} the functional \eqref{E:f} is minimized on \eqref{N:set}, which is \emph{not} a manifold); whereas, in the superlinear case, the dual formulation allows to minimize in a Nehari \emph{manifold} in a Banach space. 

The use of a direct or a dual approach has, of course, a strong influence on the methods to study qualitative properties of solutions; this is particularly clear in the proof of the symmetry breaking, described above, where our proof relies on a transformation in dual spaces to obtain monotonicity and then on comparison principles to control the effects of the nonlocal operator $K$.  The proof of the symmetry breaking result in \cite{PW15} (although is also done by contradiction finding a direction along which the energy would decrease) is very different from ours, and it relies on a unique continuation property for minimizers, which is obtained using known uniqueness results \cite{Wong75} and nonoscillation criteria \cite{Gollwitzer70} for sublinear ODEs of type $y''+a(t)|y|^{p-1}y=0$. These general theorems are not known for systems of ODEs, and in fact, they may fail in general. Using elementary manipulations in an ODE setting, one can find some extensions of these techniques to systems. Although we do not need these results for any of our proofs, we believe they can be of independent interest; in particular, we use them to show a result on the simplicity of zeros of \emph{any} radial solution that satisfies a unique continuation property. 
\begin{theo}\label{aux:lemma:intro}
Let $p,q>0$, $N\geq 2$, $\Omega$, $\delta$ as in \eqref{Omega:eqs}, and $(u,v)\in [C^{2,\varepsilon}(\overline{\Omega})]^2$ be a radial classical solution of \eqref{NHS} such that $u^{-1}(0)\cap v^{-1}(0)$ has empty interior. Then
\begin{align}\label{Weps}
u,v\in W:=\{ w\in C^1(\overline{\Omega})\::\: \nabla w(x) \neq 0 \text{ if } x\in\overline{\Omega} \text{ satisfies } |x|>\delta \text{ and } w(x)=0\},
\end{align}
\end{theo}
This theorem yields that, if $(u,v)$ is a radial solution of \eqref{NHS} satisfying a (weak) unique continuation property, then the nodal set of $u$ and of $v$ is at most a countable union of spheres and thus $|u^{-1}(0)|=|v^{-1}(0)|=0$.

\medskip

To close this introduction, we mention some open questions about system \eqref{NHS}: it is unknown whether or not least energy solutions are (up to rotations) \emph{unique}; in fact, uniqueness is open even for least energy radial solutions. Moreover, when $pq<1$ it is unclear if the unique continuation property (\emph{i.e.}, that the nodal sets of solutions have zero Lebesgue measure) holds in general for \emph{all} solutions, and not just for minimizers. %These questions are open even for the single equation \eqref{eq:1eq_problem}. 
When $pq>1$, with the exception of the single equation case $p=q$, the unique continuation property  is an open question even for least energy solutions. In Remark \ref{rem:superlinear_UCP} below we observe that, in the superlinear case, if a least energy solution $(u,v)$ vanishes on a set of positive measure, then it has a zero of infinite order. For single equations of type $-\Delta u+V(x)u=0$ such a result is available for a large class of potentials (see\cite[Proposition 3]{FG92} and \cite[Proposition 1.1]{GL93}) and  having a zero of infinite order typically implies that a solution is identically zero (see \cite{JK85}).

\medskip

The paper is organized as follows. In Section \ref{sec:variational} we provide some general auxiliary results in our dual framework; in particular, we show Lemmas \ref{minimum} and \ref{lemma:B}, which prove the existence and regularity claims in Theorem \ref{thm:maintheorem1}, as well as Lemma \ref{lemma:p=q} which shows (together with Theorems~\ref{thm:maintheorem1} and \ref{thm:maintheorem3}) Corollary \ref{thm:singleeq}.  Section \ref{SEC:monotonicity} is devoted to the study of monotonicity properties of radial minimizers, here the full details of the construction of the $\divideontimes$-transformation can be found as well as the proof of Theorems \ref{thm:maintheorem3} and \ref{thm:mono:intro}. These results are then used in Section \ref{sec:SSBM} to prove the symmetry breaking result in Theorem \ref{thm:maintheorem1} as well as the foliated Schwarz symmetry claim. Finally, in Section \ref{FR:sec} we show Theorem \ref{aux:lemma:intro} as well as the unique continuation statement in Theorem \ref{thm:maintheorem1}.

\subsection{Notation}\label{N:sec}

Throughout the paper $p$ and $q$ always denote the exponents in \eqref{NHS}. We divide our proofs in two groups:
\begin{align}
 \text{Sublinear: } pq&<1 \label{sub},\\
 \text{Superlinear and subcritical: } pq&>1\quad \text{ and }\quad \frac{1}{p+1}+\frac{1}{q+1}>\frac{N-2}{N} \label{super}.
\end{align}
Note that $pq<1$ readily implies the subcriticality condition, since $pq<1$ implies
$1/(p+1)+1/(q+1)>1>(N-2)/N$. If $N=1,2$ the subcriticality condition is satisfied by any $p,q>0$. We fix
\begin{align}\label{not}
\alpha:=\frac{p+1}{p},\qquad \alpha'=p+1,\qquad \beta:=\frac{q+1}{q},\qquad \beta'=q+1.
\end{align}
Note that $\alpha'$ and $\beta'$ are the corresponding conjugate exponents of $\alpha$ and $\beta$, that is $\alpha^{-1}+(\alpha')^{-1}=1$ and $\beta^{-1}+(\beta')^{-1}=1$. 

Moreover, we define $X$, $X_s$, $K$, and $K_s$ as in the Introduction. For $s\geq 1$, $\|\cdot\|_s$ and $\|\cdot\|_{2,s}$ are the standard norms in $L^s(\Omega)$ and $W^{2,s}(\Omega)$ respectively and, if $s>1$ is such that $N>2s$, then we denote by $s^*:=\frac{Ns}{N-2s}>s$ the critical Sobolev exponent, that is, $s^*$ is the biggest exponent which allows the (continuous) embedding $W^{2,s}(\Omega)\hookrightarrow L^{s^*}(\Omega)$. If $N=1,2$ we set $s^*=\infty$.

If $Y$ is a vector space, then $[Y]^2:=Y\times Y$. We denote by $B_r$ the ball in $\R^N$ of radius $r>0$ centered at 0. 
Finally, the nodal set of $u:\Omega\to \R$ is denoted by $u^{-1}(0):=\{x\in \Omega\::\: u(x)=0\}$.

For a measurable set $A\subset\R^N$, $|A|$ denotes its Lebesgue measure, and the function $\chi_A$ is the characteristic function of $A$, that is, $\chi_A(x)=1$ if $x\in A$ and $\chi_A(x)=0$ if $x\not\in A$. Finally, we use $\omega_N$ to denote the measure of the unitary ball in $\R^N$.

\section{Variational framework and existence results}\label{sec:variational}

In this section we describe our dual approach. In the following $\Omega$ is a smooth bounded domain in $\R^N$, $p$ and $q$ satisfy \eqref{sc:intro}, and let $X^s$, $X$, $K$, $K_t$, $\alpha$, and $\beta$ be as above.  

\subsection{The dual method}
The following lemma recalls some well-known regularity for Neumann problems, see for example \cite[Theorem and Lemma in page 143]{RR85} (see also \cite[Theorem 15.2]{ADN59}).
\begin{lemma}\label{l:reg}
 If $s>1$, $\Omega$ be a smooth bounded domain in $\R^N$, and $h\in L^s(\Omega)$ with $\int_\Omega h =0$, then there is a unique strong solution $u\in W^{2,s}(\Omega)$ of 
 \begin{align}\label{Nprob}
  -\Delta u = h\quad \text{ in }\Omega,\qquad \partial_\nu u=0\quad \text{ on }\partial \Omega,\qquad \int_\Omega u = 0
  \end{align}
  in particular,
 \begin{align*}
  \int_\Omega \nabla u\nabla \varphi = \int_\Omega h \varphi \qquad \text{ for all }\varphi\in W^{1,s'}(\Omega),\ s'=\frac{s}{s-1},
 \end{align*}
 and there is $C(\Omega,s)=C>0$ such that $\|u\|_{2,s}\leq C\|h\|_s.$
\end{lemma}

Recall that if $f\in L^s(\Omega)$ for some $s>1$, then $Kf$ is the solution of \eqref{Nprob} with $h=f$. 
Observe that 
\begin{align}\label{ibyp:ab}
\int_\Omega g\, K_{q} f = \int_\Omega g\, K f = \int_\Omega Kg\, f = \int_\Omega K_{p} g\, f,
\end{align}
by integration by parts and because $\int_\Omega f=\int_\Omega g=0$. Our next result shows that $\phi$ is continuously differentiable.

\begin{lemma}\label{decom:lemma}
 The functional $\phi:X\to\R$ as defined in \eqref{phi:def} is continuously differentiable in $X$ with  
 \begin{align}\label{phip}
  \phi'(f,g)(\varphi,\psi)=\int_\Omega |f|^{\alpha-2}f\varphi+|g|^{\beta-2}g\psi-\psi\, K f-\varphi\, K g\ dx\qquad \text{ for }(f,g),(\varphi,\psi)\in X.
 \end{align}
\end{lemma}
\begin{proof}
Let $\phi=\psi-T$, where 
\begin{align}\label{decom}
\Psi(f,g):=\int_\Omega \frac{|f|^{\alpha}}{\alpha}+\frac{|g|^{\beta}}{\beta}\, dx \qquad \text{ and }\qquad T(f,g):=\int_\Omega g\, K f,\qquad (f,g)\in X.
 \end{align}
Since $\alpha,\beta>1$ it is standard to show that $\Psi$ is continuously differentiable with
\begin{align}\label{psip}
 \Psi'(f,g)(\varphi,\psi)=\int_\Omega |f|^{\alpha-2}f\varphi+|g|^{\beta-2}g\psi\, dx, \qquad (f,g),(\varphi,\psi)\in X.
\end{align}
Now we show that $T$ is bilinear and bounded. Indeed, this follows directly from the fact that $K(\lambda h) = \lambda K h$ for $h\in L^s(\Omega)$, $s\in\{\alpha,\beta\}$, and the following integration by parts 
\begin{align}\label{ibyp}
T(f,g)=\int_\Omega g K f = \int_\Omega f K g,\qquad (f,g)\in X.
\end{align}
%where we used that $\int_\Omega f=\int_\Omega g = 0$. 
Moreover, by H\"{o}lder's inequality, Lemma \ref{l:reg}, and the first embedding in \eqref{embed}, there are $C_1,C_2>0$ such that
\begin{align}\label{est1}
|T(f,g)|\leq \frac{1}{2} \|g\|_{\beta} \|K f\|_{\beta'}\leq C_1\|g\|_{\beta} \|K f\|_{W^{2,\alpha}(\Omega)} \leq C_2\|g\|_{\beta} \|f\|_{\alpha}.
\end{align}
%Then, by \eqref{K:def}, $|T(f,g)|\leq C \|f\|_{\alpha}\|g\|_{\beta}$ for some $C>0$. 
Thus $T$ is a (continuously differentiable) bilinear bounded mapping and, by \eqref{ibyp},
\begin{align}\label{Tp}
 T'(f,g)(\varphi,\psi)=\int_\Omega \psi\, K f+\varphi\, K g \, dx, \qquad (f,g),(\varphi,\psi)\in X.
\end{align}
Therefore, \eqref{phip} follows from \eqref{psip} and \eqref{Tp}.
\end{proof}

Before we argue existence of solutions, we establish a one-to-one relationship between critical points of $\phi$ and strong solutions of \eqref{NHS}, see \cite[Proposition 3.1]{BMT14} for the Dirichlet case.

\begin{lemma}\label{l:ss}
An element $(f,g)\in X$ is a critical point of $\phi$, i.e.,
\begin{align}\label{cp}
\phi'(f,g)(\varphi,\psi)=0\qquad \text{ for all }(\varphi,\psi)\in X,
\end{align}
if and only if $(u,v):=(K_{p} g,K_{q} f)\in W^{2,\beta}(\Omega)\times W^{2,\alpha}(\Omega)$ is a strong solution of \eqref{NHS}, that is, $(u,v)$ solves \eqref{NHS} a.e. in $\Omega$.
\end{lemma}
\begin{proof}
Let $(f,g)\in X$ satisfy \eqref{cp} and let $(u,v):=(K_{p} g,K_{q} f)$. Then, using \eqref{ibyp:ab},
\begin{align*}
\int_\Omega (|f|^{\alpha-2}f- u)\varphi+(|g|^{\beta-2}g-v)\psi \ dx = 0\qquad \text { for all }(\varphi,\psi)\in X=X^\alpha\times X^\beta.
\end{align*}
Let $w:=|f|^{\alpha-2}f- u$ and, for a function $\zeta\in L^\alpha(\Omega)$, set $\overline\zeta:=|\Omega|^{-1}\int_\Omega \zeta\in\R$. Then $(\zeta-\overline\zeta)\in X^\alpha$ and
\begin{align*}
0=\int_\Omega w(\zeta-\overline\zeta) = \int_\Omega w\zeta-\overline\zeta \int_\Omega w=\int_\Omega w\zeta-\overline w \int_\Omega \zeta = \int_\Omega (w - \overline w)\zeta\qquad \text{ for all }\zeta\in L^\alpha(\Omega),
\end{align*}
and therefore $w-\overline w=0$ a.e. in $\Omega$, which implies that $|f|^{\alpha-2}f=u+\overline w$ a.e. in $\Omega$ and thus $f = |u+\overline w|^{p-1}(u+\overline w)$ a.e. in $\Omega$.  Furthermore, since $(f,g)\in X$ and $u=K_{p} g$, we have that $\int_\Omega|u|^{p-1} u = 0 = \int_\Omega f = \int_\Omega|u+\overline w|^{p-1}(u+\overline w)$, which implies that $\overline w=0$, and therefore $-\Delta v = -\Delta (K_{q} f) = f = |u|^{p-1}u$ a.e. in $\Omega$. Analogously, $-\Delta u = |v|^{q-1}v$ a.e. in $\Omega$, as claimed.  

\medskip

For the converse implication, let $(u,v):=(K_{p} g,K_{q} f)\in W^{2,\beta}(\Omega)\times W^{2,\alpha}(\Omega)$ be a strong solution of \eqref{NHS} for some $(f,g)\in X$. Then, necessarily $f=|u|^{p-1}u$ and $g=|v|^{q-1}v$ a.e. in $\Omega$, which implies that $u=|f|^{\alpha-2}f$ and $v=|g|^{\beta-2}g$ a.e. in $\Omega$, which implies \eqref{cp}, by \eqref{ibyp:ab}.

\end{proof}

The next Proposition states that critical points of $\phi$ are in fact classical solutions. We refer to \cite[Theorem 1]{S00} and \cite[Lemma 5.16]{BMT14} for analogous results in the superlinear Dirichlet case. 
\begin{prop}\label{p:reg}
 Let $(f,g)\in X$ be a critical point of $\phi$, then $(u,v):=(K_{p} g,K_{q} f)\in [C^{2,\varepsilon}(\overline{\Omega})]^2$ for some $\varepsilon>0$ and satisfies \eqref{NHS} pointwise.
\end{prop}
\begin{proof}
The proof follows closely \cite[Lemma 5.16]{BMT14}.  For $t,s\geq 1$ denote
\begin{align*}
 W(s,t):=W^{2,s}(\Omega)\times W^{2,t}(\Omega)\qquad \text{ and }\qquad L(s,t):=L^s(\Omega)\times L^t(\Omega).
\end{align*}
Let $(f,g)\in X \subset L(\alpha,\beta)$ be a critical point of $\phi$ then, by Lemma \ref{l:ss}, $(u,v)\in W(\beta,\alpha)\hookrightarrow L(\beta^*,\alpha^*)$ is a strong solution of \eqref{NHS}.  For $n\in \N_0,$ let 
\begin{align}\label{hyp}
(\beta_0,\alpha_0):=(\beta,\alpha),\quad (\beta_{n+1},\alpha_{n+1}):=(\,\frac{\alpha_{n}^*}{q},\frac{\beta_{n}^*}{p}\,) \quad\text{ if }\quad N>2\alpha_{n}, \quad N>2\beta_{n}.
\end{align}
Here we are using the notation given in Section \ref{N:sec} for the critical Sobolev exponent. Note that $(|v|^{q},|u|^{p})\in L(\alpha^*/q,\beta^*/p)=L(\beta_1,\alpha_1)$ and then Lemma \ref{l:reg} and \eqref{NHS} imply $(u,v)\in W(\beta_1,\alpha_1)\hookrightarrow L(\beta^*_1,\alpha_1^*)$. But then
$(|v|^{q},|u|^{p})\in L(\beta_2,\alpha_2)$, which gives $(u,v)\in W(\beta_2,\alpha_2)\hookrightarrow L(\beta^*_2,\alpha_2^*)$. Iterating this procedure we obtain that $(u,v)\in W(\beta_n,\alpha_n)$ as long as $N>2\alpha_{n-1}$ and $N>2\beta_{n-1}$.  

We claim that 
\begin{align}\label{goal}
N\leq 2s\qquad \text{ for some }\quad s\in\{\alpha_n,\beta_n : n\in\N_0\}, 
\end{align}
and then the proposition follows from Lemma \ref{l:ss}, the embedding $W^{2,s}(\Omega)\hookrightarrow C^\mu(\Omega)$ for some $\mu>0$, and standard Schauder estimates (see \cite[Theorem 6.31]{GT98} and the Remark after the theorem), since $t\mapsto |t|^{r-1}t$ for $r>0$ is H\"{o}lder continuous.

Indeed, assume by contradiction that $N> 2s$ for all $s\in\{\alpha_n,\beta_n : n\in\N_0\}$ and consider the sequence $S_n:=(q\beta_n,p\alpha_n)$. By the subcriticality assumption we have that $W(\beta,\alpha)\hookrightarrow L(\alpha',\beta')=L(p\alpha,q\beta)$ and therefore  $S_1=(q\beta_1,p\alpha_1)=(\alpha^*,\beta^*)>(q\beta,p\alpha)=S_0$
and $(S_n)_{n\in \N}$ is, by induction, a (component-wise) increasing (bounded) sequence. Let $l_1,l_2>0$ be such that $S_n\to (l_1,l_2)$ and note that, by the monotonicity,
\begin{align}\label{c}
 l_2 > p\alpha = p+1.
\end{align}
Moreover
\begin{align*}
 (l_1,l_2)=\lim_{n\to\infty}(\alpha_{n}^*,\beta_{n}^*)=(\frac{Nl_2}{Np-2l_2},\frac{Nl_1}{Nq-2l_1}),\quad \text{ which implies }\quad  l_2=\frac{N(pq-1)}{2(q+1)}.
\end{align*}
Since the function $t\mapsto h(t):=\frac{N(pt-1)}{2(t+1)}$ is increasing for $t>0$ and $q<t_0:=\frac{2p+N+2}{pN-2p-2}$, by the subcriticality condition \eqref{super}, we have that $l_2=
h(q)<h(t_0)=p+1$ , which contradicts \eqref{c}. Then \eqref{goal} holds and this ends the proof.
\end{proof}

The next Lemma shows the relationship between the dual and the direct energy functionals when evaluated on \emph{solutions}. 
\begin{lemma}\label{l:same}
 Let $(f,g)\in X$ be a critical point of $\phi$ and let $(u,v):=(K_{p} g,K_{q} f)$.  Then
 \begin{align}\label{I:def}
  \phi(f,g)=I(u,v):=\int_\Omega \nabla u\nabla v - \frac{|u|^{p+1}}{p+1}-\frac{|v|^{q+1}}{q+1}\, dx
 \end{align}
\end{lemma}
\begin{proof}
Let $f,g,u$ and $v$ as in the statement. By Proposition \ref{p:reg}, we have that $(u,v)$ is a solution of \eqref{NHS} and, integrating by parts,
 \begin{align*}
  \phi(f,g)&=\int_\Omega \frac{|f|^\alpha}{\alpha}+\frac{|g|^\beta}{\beta}-g K f \, dx= \int_\Omega p\frac{|u|^{p+1}}{p+1}+q\frac{|v|^{q+1}}{q+1} - \nabla u \nabla v \, dx.
 \end{align*}
Therefore $\phi(f,g)-I(u,v)= \int_\Omega |u|^{p+1}+|v|^{q+1}- 2\nabla u \nabla v\, dx=0$, by \eqref{NHS}.
\end{proof}

We finish this Section with a result in the case $p=q$, in which \eqref{NHS} reduces to a single equation.  The proof is the same as in the superlinear Dirichlet case \cite[Theorem 1.5]{BMRT15} and we include it for the reader's convenience.
\begin{lemma}\label{lemma:p=q}
 Let $p=q>0$ and $(u,v)\in C^{2}(\overline{\Omega})\times C^{2}(\overline{\Omega})$ be a classical solution to \eqref{NHS}, then $u=v$.
\end{lemma}
\begin{proof}
 By testing \eqref{NHS} with $u,v$ and integrating by parts we have that
 \begin{align*}
\int_\Omega |\nabla u|^2=\int_\Omega |v|^{p-1}v\, u,
  \quad \int_\Omega |\nabla v|^2 = \int_\Omega |u|^{p-1}u\, v,
  \quad \text{ and }\quad 
  \|v\|_{\alpha'}^{\alpha'}=\int_\Omega \nabla u\nabla v =\|u\|_{\alpha'}^{\alpha'}.
   \end{align*}
Then, by H\"{o}lder's inequality,
\begin{align*}
\int_\Omega |\nabla u|^2+\int_\Omega |\nabla v|^2
\leq \|u\|_{\alpha'}\|v\|_{\alpha'}^p+\|v\|_{\alpha'}\|u\|_{\alpha'}^p=2\int_\Omega \nabla u\nabla v,
\end{align*}
which implies $\nabla u=\nabla v$ in $\Omega$ and then $v-u\equiv c$ in $\Omega$ for some constant $c\in\R$.  Therefore, by \eqref{NHS}, $|u|^{p-1}u=-\Delta u=-\Delta (u+c) =|u+c|^{p-1}(u+c)$ in $B$. If $u\equiv 0$ then also $v\equiv 0$. If $u\not \equiv 0$ then $u$ changes sign and there is $x\in B$ with $u(x)=0$, but then $|c|^{p-1}c=0$, and $u=v$ in $\Omega$ as claimed.
% To see that $c=0$, note that, integrating by parts \eqref{NHS}, we have
% \begin{align}\label{zn}
%  \int_\Omega |u|^{p-1}u=\int_\Omega |u+c|^{p-1}(u+c)=0.
% \end{align}
% Note that the function $h(t,x):=|u(x)+t|^{p-1}(u(x)+t)$ is strictly increasing in $t$ for all $x\in\Omega$ (since $h_t(t,x)=p|u(x)+t|^{p-1}>0$ if $u(x)+t\neq 0$). Then, if $c>0$, we have $h(0,x)<h(c,x)$ for all $x\in\Omega$, which would contradict \eqref{zn}.  Similarly, if $c<0$, then $h(0,x)>h(c,x)$ for all $x\in\Omega$ also contradicting \eqref{zn}. Therefore $c=0$ and $u=v$ in $\Omega$ as claimed.
\end{proof}

\subsection{Existence of least energy solutions: sublinear case}

Assume \eqref{sub}, that is, $p,q>0$ and $pq<1$ and recall the notation \eqref{not}.  

\begin{lemma}\label{l:coer}
 The functional $\phi:X\to\R$ as defined in \eqref{phi:def} is coercive.
\end{lemma}
\begin{proof}
Let $\varepsilon:=\frac{1}{2}\min\{\alpha^{-1},\beta^{-1}\}$. Then from \eqref{est1}, Young's inequality, and the fact that $\beta'<\alpha$ (since $pq<1$) there is $C_2(\alpha,\beta)=C_2>0$ such that $|T(f,g)|\leq \varepsilon(\|f\|^\alpha_{\alpha}+\|g\|^\beta_{\beta})+C_2.$ This yields that 
\begin{align}\label{coer}
\phi(f,g)\geq \varepsilon(\|f\|^\alpha_{\alpha}+\|g\|^\beta_{\beta})-C_2, 
\end{align}
that is, $\phi$ is coercive.
\end{proof}

\begin{lemma}\label{minimum}
The functional $\phi$ achieves a negative minimum in $X$, that is,
\begin{align*}
 \min_X \phi = \phi(f,g)<0\qquad \text{for some }(f,g)\in X.
\end{align*}
Moreover, $\phi'(f,g)=0$ in $X$, $(u,v):=(K_{p} g, K_{q} f)\in [C^{2,\varepsilon}(\overline \Omega)]^2$ is a classical solution of \eqref{NHS}, and $(u,v)$ is a least energy solution, that is, $I(u,v)=c$ with $c$ as in \eqref{c:level}.
\end{lemma}
\begin{proof}
For $n\in \N$ let $x_n:=(f_n,g_n)$ be a minimizing sequence in $X$, which exists by Lemma~\ref{l:coer}.  Since $X$ is reflexive there is $x:=(f,g)\in X$ such that $x_n\rightharpoonup x$ weakly in $X$ and $(x_n)$ is bounded in $X$ in virtue of \eqref{coer}. Since the operator $K$ is compact by Lemma \ref{l:reg} and the compact embedding \eqref{embed}, we have that $K(f_n)$ converges strongly to $Kf$ in $L^{\beta'}(\Omega)$. Using the lower semicontinuity of norms we have that $\min\limits_X\phi=\liminf\limits_{n\to\infty}\phi(f_n,g_n)\geq \phi(f,g)$, and therefore $\phi$ achieves its minimum.

To see that the minimum is strictly negative, let $T$ and $\Psi$ as in \eqref{decom} and let $\varphi\in C^\infty_c(\Omega)\backslash\{0\}$ such that $\int_\Omega \varphi = 0$. Then $(\varphi,\varphi)\in X\backslash\{(0,0)\}$ and
\begin{align*}
T(\varphi,\varphi)=\int_\Omega \varphi K \varphi =\int_\Omega (-\Delta u) u =  \int_\Omega |\nabla u|^2> 0,\qquad \text{ where }u:=K \varphi.
\end{align*}
Thus, since $pq<1$ is equivalent to $\alpha\beta>\alpha+\beta$,
\begin{align}\label{phi:neg}
 \phi(t^{\frac{\beta}{\alpha+\beta}}\varphi,t^{\frac{\alpha}{\alpha+\beta}}\varphi)
 = t^{\frac{\alpha\beta}{\alpha+\beta}}\Psi(\varphi,\varphi)-tT(\varphi,\varphi)<0\qquad \text{for $t>0$ small enough.}
\end{align}
Then $\phi'(f,g)=0$ in $X$ by Lemma \ref{decom:lemma}, $(u,v):=(K_{p} g, K_{q} f)\in [C^{2,\varepsilon}(\overline \Omega)]^2$ is a classical solution of \eqref{NHS}, by Proposition \ref{p:reg}, and $(u,v)$ is a least energy solution, by Lemma \ref{l:same}.
\end{proof}

\subsection{Existence of least energy solutions: superlinear case}   

Let $p,q\in(0,\infty)$ such that \eqref{super} holds.
\begin{align}\label{gammas}
\gamma_1:=\frac{\beta}{\alpha+\beta},\qquad \gamma_2:=\frac{\alpha}{\alpha+\beta},\qquad \text{ and }\qquad \gamma:=\frac{\alpha\beta}{\alpha+\beta}<1,
\end{align}
where $\gamma<1$ because $pq>1$. In particular, $\gamma_1\alpha = \gamma_2\beta=\gamma$ and $\gamma_1+\gamma_2=1$. We define the Nehari-type set
\begin{align}
\mathcal{N}&:=\{(f,g)\in X\backslash \{(0,0)\} : \ \phi'(f,g)(\gamma_1f,\gamma_2g)=0\}\nonumber\\
		&=\left\{(f,g)\in X\backslash \{(0,0)\} : \ \int_\Omega \gamma_1 |f|^\alpha + \gamma_2 |g|^\beta\, dx=\int_\Omega f Kg\right\}.\label{page13}
\end{align}

\begin{lemma}\label{lemma:A}
 Let $(f,g)\in X\backslash\{(0,0)\}$ such that $\int_\Omega f K g>0$, then 
 \begin{align}\label{eq:projection}
  t:=t({f,g}):=\Big(\frac{\gamma_1 \int_\Omega |f|^\alpha + \gamma_2 \int_\Omega |g|^\beta}{\int_\Omega fKg}\Big)^\frac{1}{1-\gamma}
 \end{align}
is the unique maximum of the function $s\mapsto \phi(s^{\gamma_1}f,s^{\gamma_2}g)$ and $(t^{\gamma_1}f,t^{\gamma_2}g)\in \cN$.
\end{lemma}
\begin{proof}
Since
$
\phi(s^{\gamma_1}f,s^{\gamma_2})=s^\gamma \int_\Omega \frac{|f|^\alpha}{\alpha}+\frac{|g|^\beta}{\beta}\, dx-s\int_\Omega f K g,
$
$\gamma<1$, and $\int_\Omega f K g>0$, the map $s\mapsto \phi(s^{\gamma_1}f,s^{\gamma_2}g)$ has a unique positive critical point, which is a global maximum. The claim now follows by direct computations.
\end{proof}

\begin{lemma}\label{lemma:B}
 There is $(f,g)\in\cN$ such that $\phi(f,g)=\inf_\cN \phi$. Moreover, $\phi'(f,g)=0$ in $X$, $(u,v):=(K_{p} g, K_{q} f)\in [C^{2,\varepsilon}(\overline \Omega)]^2$ is a classical solution of \eqref{NHS}, and $(u,v)$ is a least energy solution, that is, $I(u,v)=c$ with $c$ as in \eqref{c:level}.
\end{lemma}
\begin{proof} 1) The functional $\phi$ is bounded from below on $\mathcal{N}$, since
\begin{align}\label{eq:phi_in_N}
\phi(f,g)=\frac{1-\gamma}{\alpha}\int_\Omega |f|^\gamma+\frac{1-\gamma}{\beta}\int_\Omega |g|^\beta\geq 0\qquad \text{ for } (f,g)\in \mathcal{N}.
\end{align}
Thus $\inf_\mathcal{N} \phi \in \R$.  Take a minimizing sequence $(f_n,g_n)\in \mathcal{N}$. From \eqref{eq:phi_in_N}, we have that $\|f_n\|_\alpha$ and $\|g_n\|_\beta$ are bounded and thus, up to a subsequence, $f_n\rightharpoonup f$ in $L^\alpha$, $g_n\rightharpoonup g$ in $L^\beta$, for some $(f,g)\in X$.
\smallbreak

\noindent 2) Recalling \eqref{est1} and using Young's inequality, we have
\[
\gamma_1 \|f_n\|_\alpha^\alpha+\gamma_2 \|g_n\|_\beta^\beta=\int_\Omega f_nKg_n \leq C\|f_n\|_\alpha \|g_n\|_\beta\leq C(\gamma_1 \|f_n\|_\alpha^\frac{1}{\gamma_1}+\gamma_2\|g_n\|_\beta^\frac{1}{\gamma_2})
\]
for some $C>0$ and, since $\frac{1}{\gamma_1}>\alpha$, $\frac{1}{\gamma_2}>\beta$, there exists $\delta>0$ such that
$\int_\Omega f_nKg_n \geq \delta$  for every $n$.

\smallbreak

\noindent 3) Combining this with the properties of weak convergence and the compactness of the operator $K$, we have $(f,g)\neq (0,0)$ and, by definition of $\mathcal{N}$,
\[
0<\int_\Omega \gamma_1 |f|^\alpha+\gamma_2 |g|^\beta\, dx\leq \int_\Omega f Kg.
\]
Therefore, by Lemma \ref{lemma:A}, there exists $0<t\leq 1$ such that $(t^{\gamma_1} f, t^{\gamma_2}g)\in \mathcal{N}$. Then,
\begin{align*}
\inf_\cN \phi &\leq \phi(t^{\gamma_1} f , t^{\gamma_2} g )  =  t^\gamma (1-\gamma) \int_\Omega \frac{|f|^\alpha}{\alpha}+\frac{|g|^\beta}{\beta}\, dx\\
		&\leq (1-\gamma) \int_\Omega \frac{|f|^\alpha}{\alpha}+\frac{|g|^\beta}{\beta}\, dx
\leq \liminf_{n\to\infty}\phi(f_n,g_n)=\inf_\cN \phi,
\end{align*}
thus $t=1$, and $(f,g)\in \mathcal{N}$ and achieves $\inf_\mathcal{N}\phi$.

\smallbreak

\noindent 4) Defining $\tau(h,k)=\phi'(h,k)(\gamma_1h, \gamma_2 k)$ for $(h,k)\in \mathcal{N}$, we have
\begin{align}
\tau'(h,k)(\gamma_1 h,\gamma_2 k)=\gamma_1(\gamma-1) \int_\Omega |h|^\alpha + \gamma_2 (\gamma-1) \int_\Omega |k|^\beta<0.
\end{align}
Therefore $\mathcal{N}$ is a manifold and, since $(f,g)$ achieves $\inf_\mathcal{N} \phi$, then by Lagrange's multiplier rule there exists $\lambda\in \R$ such that $\phi'(f,g)=\lambda \tau'(f,g)$. By testing this identity with $(\gamma_1 f,\gamma_2 g)$, we see that actually $\lambda=0$ and $(f,g)$ is a critical point of $\phi$ in $X$. Thus $(u,v):=(K_{p} g, K_{q} f)\in [C^{2,\varepsilon}(\overline \Omega)]^2$ is a classical solution of \eqref{NHS} by applying Proposition \ref{p:reg}. Finally, $(u,v)$ is a least energy solution, by Lemma \ref{l:same}.
\end{proof}

\section{Monotonicity of radial minimizers}\label{SEC:monotonicity}
In this subsection $\Omega\subset \R^N$ is either
\begin{equation}\label{Om}
\begin{aligned}
\Omega&=B_1(0)\qquad \qquad \text{ and fix $\delta:=0$}\qquad \quad \text{(for any $N\geq 1$)}\qquad\text{ or }\\
\Omega&=B_1(0)\backslash B_\delta(0)\quad \text{ for some }\delta\in(0,1)\quad \text{(for any $N\geq 2$)}.
\end{aligned}
\end{equation}
We emphasize that the case $B_1(0)\backslash\{0\}$ is \emph{not} considered in any of our results.  Our main goal is to show Theorem \ref{thm:maintheorem3}.  This result is of independent interest, but it is also instrumental in the proof of the symmetry-breaking result in Theorem \ref{thm:maintheorem1}.  The proof of Theorem \ref{thm:maintheorem3} is based on a new transformation using rearrangement techniques and a suitable flipping compatible with our dual method approach.  We introduce first some notation and show some preliminary results.

\subsection{The decreasing rearrangement}\label{subsec:dr}

Let $U\subset \R^N$ be an open set and $h:U\to \R$ a measurable function. The \emph{decreasing rearrangement of $h$} is given by 
\begin{align*}
h^\#:[0,|U|]\to\R,\qquad h^{\#}(0):=\text{ess sup}_U h\quad h^\#(s):=\inf\{t\in\R\::\: |\{h>t\}|<s\},\ s>0.
\end{align*}
In particular, $h^\#$ is a non-increasing and left-continuous function \cite[Proposition 1.1.1]{k06}, $h^\#$ is a (level-set) rearrangement of $h$, which yields in particular that $L^p$-norms and averages are preserved, that is,
\begin{align}\label{Lp}
\|h\|_{L^p(U)}=\|h^\#\|_{L^p(0,|U|)}\qquad\text{ and }\qquad \int_U h = \int_0^{|U|} h^\#,
\end{align}
see \cite[Corollary 1.1.3]{k06}.  By \cite[Proposition 1.2.2]{k06}, we know that if $E\subset U$, then
\begin{align}\label{Ar}
 \int_E h(x) \leq \int_0^{|E|}h^\#.
\end{align}
We now focus on the case where $U=I:=[0,l]\subset \R$, for some $l>0$. If $h:I\to \R$ is non-increasing in $I$, then $h=h^\#$ a.e. in $I$, see \cite[Lemma 1.1.1]{k06}. We use the following decomposition for the integral of a product of two functions.
\begin{lemma}[Particular case of Lemma 1.2.2 in \cite{k06}]\label{cake}
Let $\varphi,\psi:I\to \R$ be measurable functions with $\varphi\in L^1(I)$, $\psi\in L^\infty(I)$, and $a,b\in\R$ with $a\leq \psi\leq b$. Then
\begin{align*}
\int_I \varphi\, \psi = a\int_I\varphi +\int_a^b \int_{\{\psi>t\}}\varphi(s)\ dsdt. 
\end{align*}
\end{lemma}
A well-known consequence of Lemma \ref{cake} and \eqref{Ar} is the Hardy-Littlewood inequality.
\begin{theo}[Particular case of Theorem 1.2.2 in \cite{k06}]\label{thm:HL}
 Let $\varphi \in L^1(I)$ and $\psi\in L^\infty(I)$ be (possibly sign-changing) functions, then
\begin{align}\label{HL}
 \int_I \varphi\, \psi \leq \int_I \varphi^\#\, \psi^{\#}
\end{align}
\end{theo}
Under some additional assumptions, the \emph{equality case} in $\eqref{HL}$ can be used to deduce monotonicity properties of the involved functions.
\begin{lemma}\label{eq:lemma}
 Let $\varphi,\psi:I\to\R$ with $\varphi\in L^1(I)$ and $\psi\in C(I)$ be a strictly decreasing function. If $\int_I \varphi\, \psi = \int_I \varphi^\#\, \psi$, then $\varphi=\varphi^\#$ a.e. in $I$.
\end{lemma}
\begin{proof}
Let $b:=\psi(0)$ and $a:=\psi(l)$. Then, by Lemma \ref{cake},
\begin{align}\label{eq1}
 a\int_I \varphi + \int_a^b \int_{\{\psi>t\}}\varphi(s)\ ds\,dt = \int_I \varphi\, \psi=\int_I \varphi^\#\, \psi=a\int_I \varphi^\# + \int_a^b \int_{\{\psi>t\}}\varphi^\#(s)\ ds\,dt.
\end{align}
Since $\psi$ is strictly decreasing and continuous, $\{\psi>t\}=[0,\psi^{-1}(t))$ for every $t\in [a,b]$. Therefore, by \eqref{Ar}, $\int_{\{\psi>t\}}\varphi(s)\ ds\leq \int_{\{\psi>t\}}\varphi^\#(s)\ ds$ for all $t\in (a,b)$ and thus, by \eqref{eq1},
\begin{align}\label{eq2}
 \int_{\{\psi>t\}}\varphi(s)\ ds = \int_{\{\psi>t\}}\varphi^\#(s)\ ds\qquad \text{ for all }t\in (a,b).
\end{align}
Furthermore, for every $t\in I$ we have $\{\psi>\psi(t)\}=[\, 0\, ,\, t\, )$, because $\psi$ is strictly monotone decreasing and continuous. Then, \eqref{eq2} yields that  
$\int_0^t\varphi = \int_0^t\varphi^\#$ for all $t\in I$, and $\varphi=\varphi^\#$ a.e. in $I$, by Lebesgue differentiation theorem.
\end{proof}

\subsection{A decreasing transformation in dual spaces}\label{sec:transformation}

In the following we do a slight abuse of notation and use $w(|x|)=w(x)$ for a radial function $w$. We use $L^\infty_{rad}(\Omega)$ and $C_{rad}(\overline \Omega)$ to denote the subspace of radial functions in $L^\infty(\Omega)$ and $C(\overline \Omega)$, respectively.  Let $\Omega$, $\delta$ as in \eqref{Om} and let
\begin{align*}
&{\cal I}:L_{rad}^\infty(\Omega)\to C_{rad}(\overline{\Omega}),\qquad {\cal I}h(x):=\int_{\{\delta{\leq}|y|{\leq}|x|\}} h(y)\ dy
=N\omega_N\int_\delta^{|x|}h(\rho)\rho^{N-1}\ d\rho\\
&\mF: C_{rad}(\overline{\Omega})\to L_{rad}^\infty(\Omega),\qquad \mF h:=(\chi_{\{{\cal I}h>0\}}-\chi_{\{{\cal I}h\leq 0\}})\, h.
\end{align*}
Then, for $h\in C_{rad}(\overline{\Omega})$, the $\divideontimes$-transformation is given by 
\begin{align*}
h^{\divideontimes}\in L^\infty_{{rad}}(\Omega),\qquad h^{\divideontimes}(x) := ({\mathfrak F} h)^\#(\omega_N |x|^N-\omega_N \delta^N).
\end{align*}
where $\#$ is the decreasing rearrangement (see the previous section) and $\omega_N= |B_1|$ is the volume of the unitary ball in $\R^N$.  The function $\mF h$ can be seen as a suitable flipping of $h$ on the set $\{\cI h\leq 0\}$, and this step is very important to construct monotone decreasing solutions, while preserving the (zero) average and the $L^p$-norm of $h$, see Proposition \ref{prop:prop} and Remark \ref{rem:in}.  See also Figure \ref{fig} below for some examples. 

\begin{figure}[p!]
\begin{subfigure}{.45\textwidth}
  \begin{picture}(150,170)
    \put(0,5){\includegraphics[width=.90\textwidth]{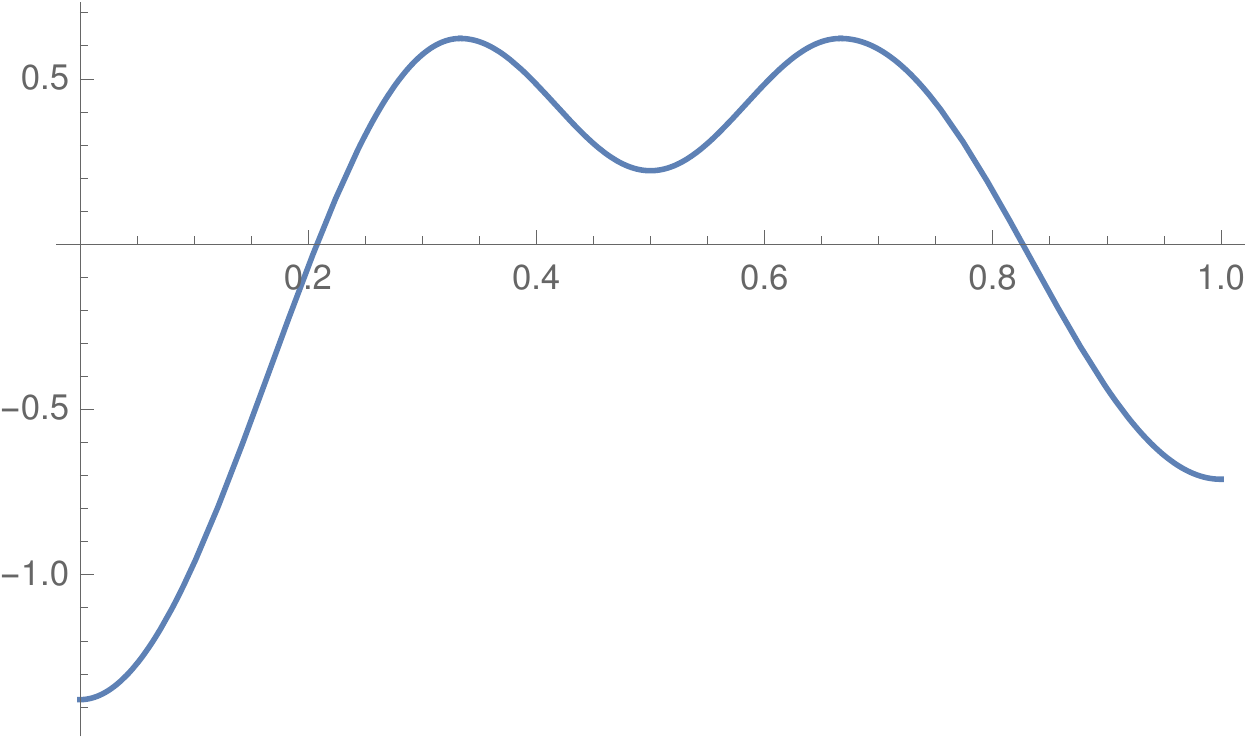}}
    \put(10,110){$h(r)$}
    \put(165,80){$r$}
  \end{picture}
\end{subfigure}%
\begin{subfigure}{.45\textwidth}
\begin{picture}(150,170)
    \put(0,5){\includegraphics[width=.90\textwidth]{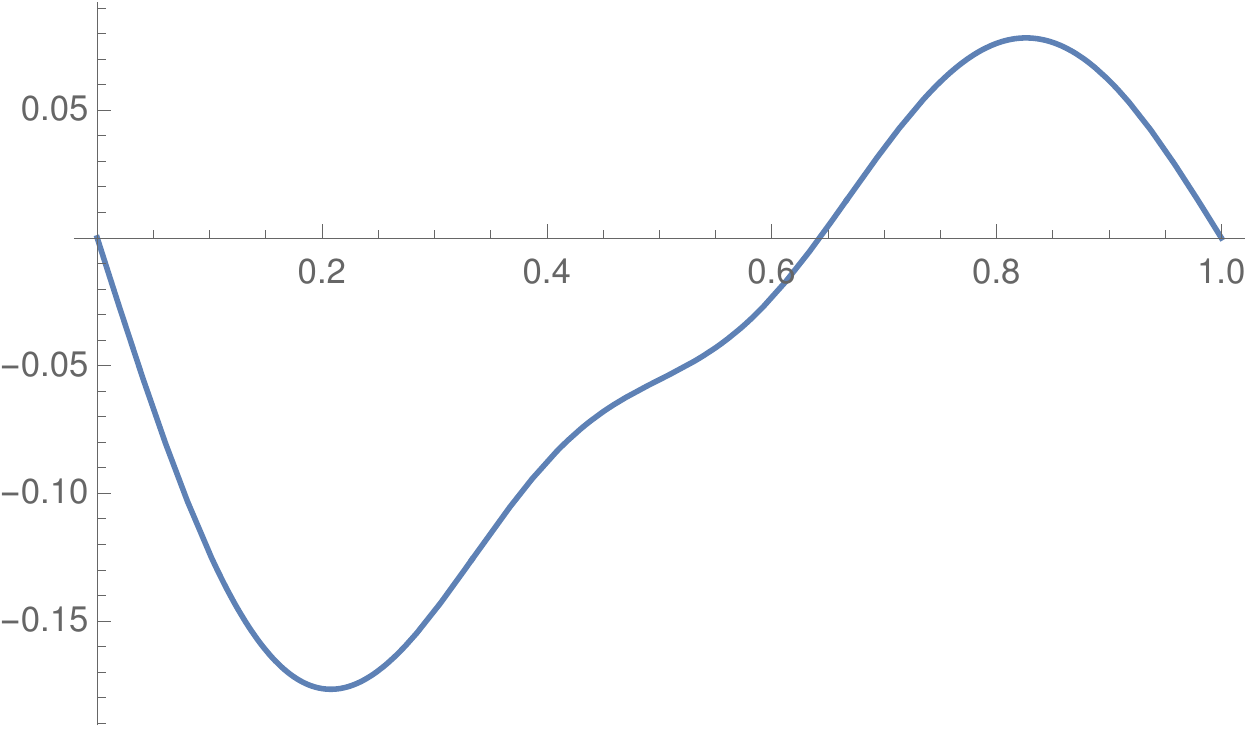}}
    \put(10,110){$\cI h(r)$}
    \put(165,80){$r$}
  \end{picture}
\end{subfigure}
\begin{subfigure}{.45\textwidth}
\begin{picture}(150,170)
    \put(0,5){\includegraphics[width=.90\textwidth]{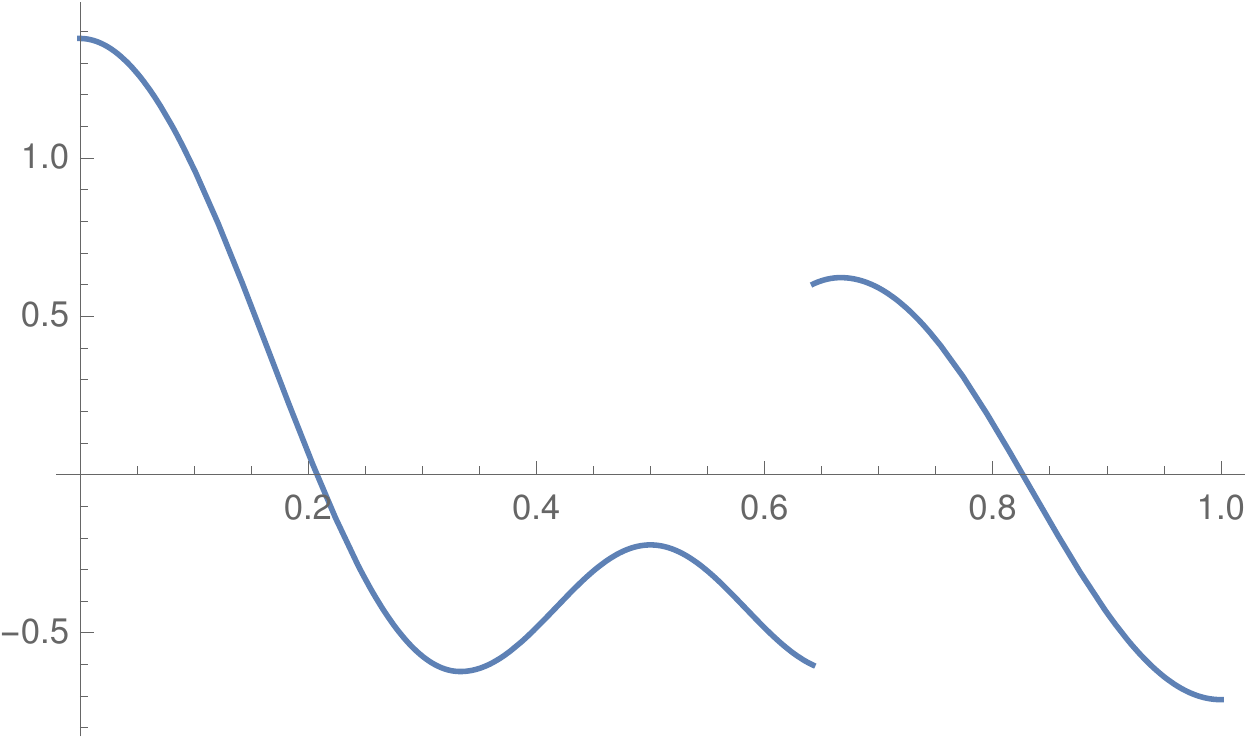}}
    \put(10,110){$\mF h(r)$}
    \put(165,50){$r$}
  \end{picture}
\end{subfigure}%
\begin{subfigure}{.45\textwidth}
\begin{picture}(150,170)
    \put(0,5){\includegraphics[width=.90\textwidth]{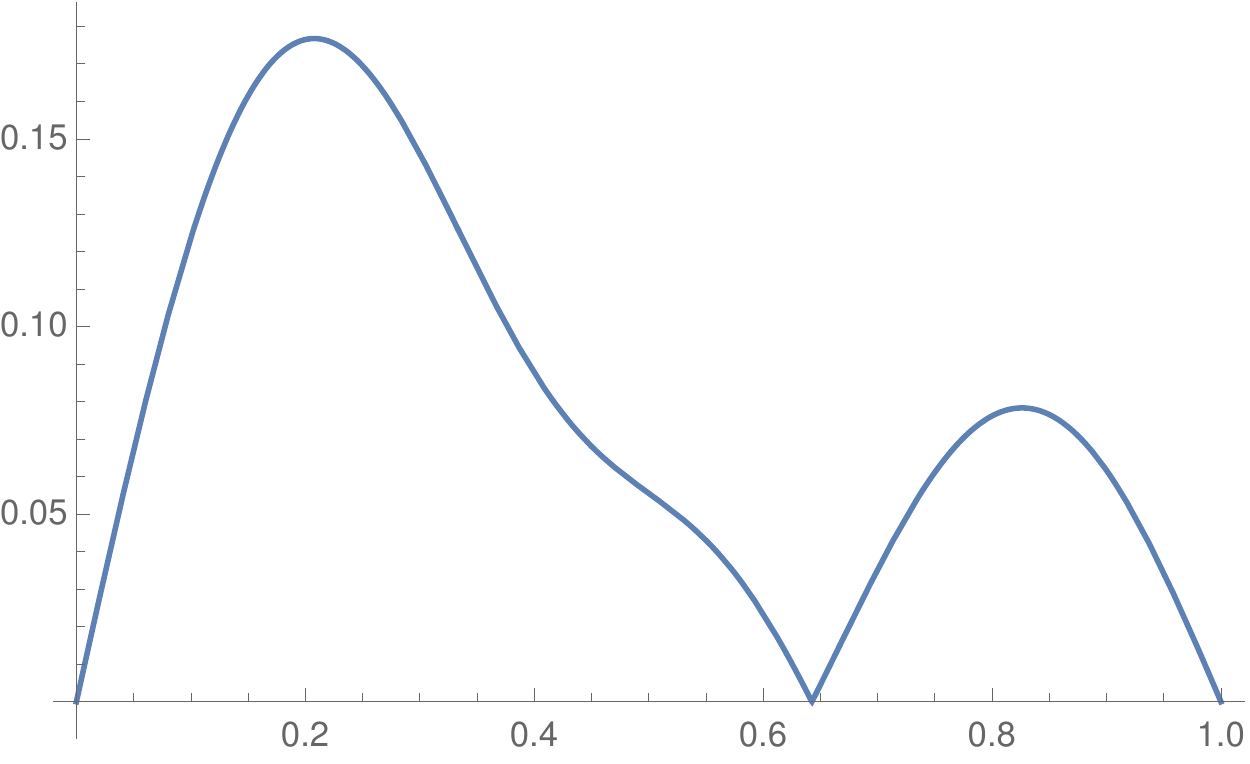}}
    \put(10,115){$\cI (\mF h)(r)$}
    \put(165,20){$r$}
  \end{picture}
\end{subfigure}
\begin{subfigure}{.45\textwidth}
\begin{picture}(150,170)
    \put(0,5){\includegraphics[width=.90\textwidth]{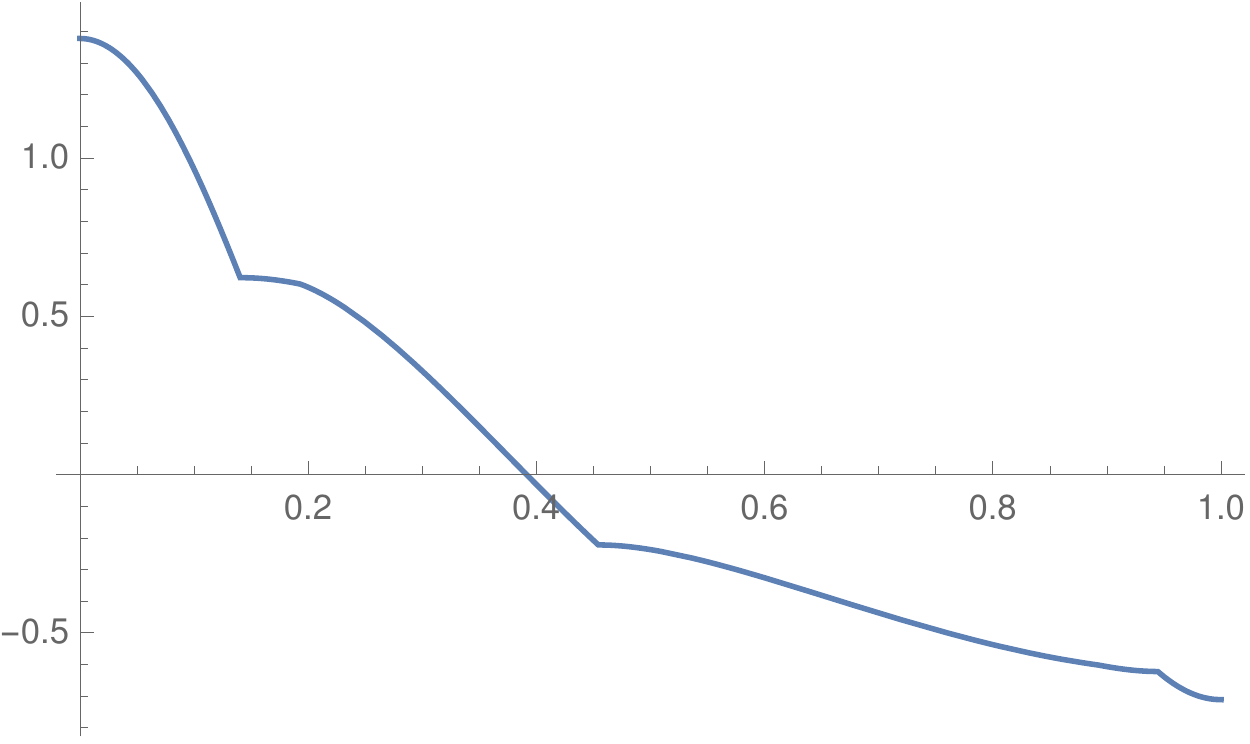}}
    \put(10,110){$h^\divideontimes(r)$}
    \put(165,50){$r$}
  \end{picture}
\end{subfigure}
\begin{subfigure}{.45\textwidth}
\begin{picture}(150,170)
    \put(0,5){\includegraphics[width=.90\textwidth]{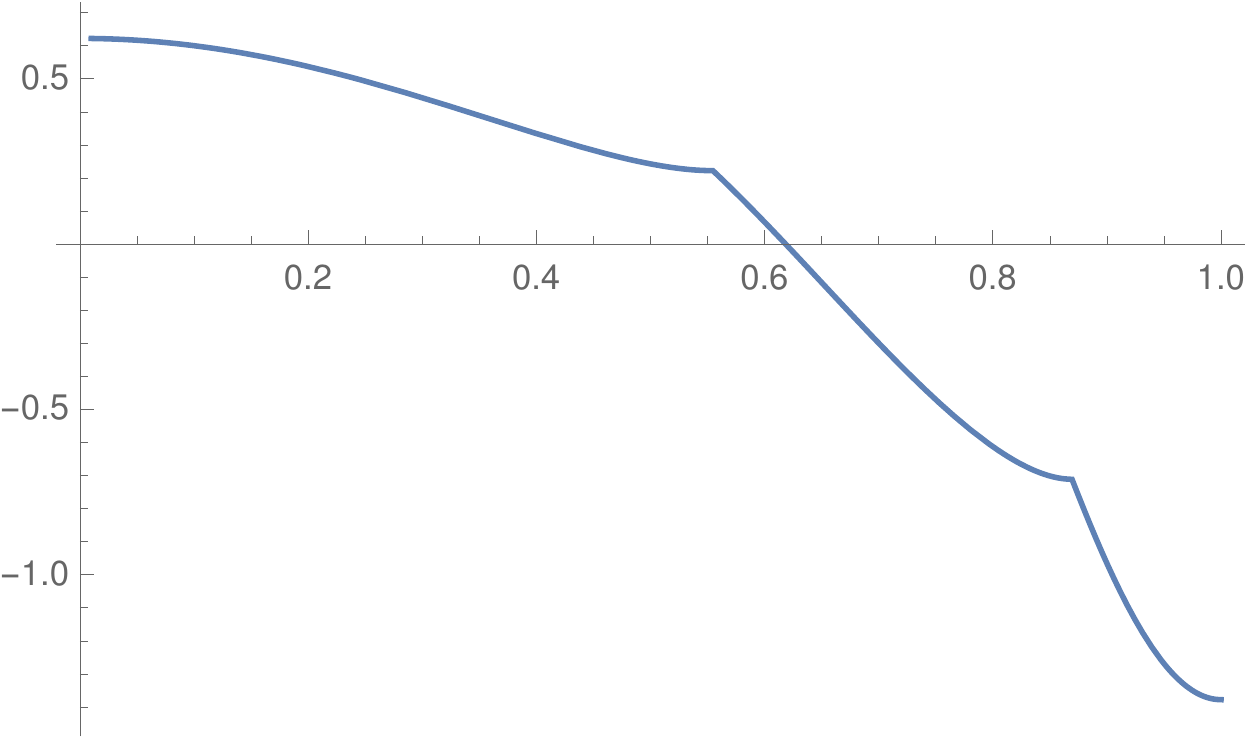}}
    \put(10,110){$h^*(r)$}
    \put(165,80){$r$}
  \end{picture}
\end{subfigure}%
\caption{Examples of the functions $\cI h$, $\mF h$, $\cI(\mF h)=|\cI h|$ for a particular radial function $h\in C(\overline{B_1(0)})$, and a comparison between $h^\divideontimes$ and the Schwarz symmetrization of $h$, $h^*(r):=h^\#(\omega_Nr^N)$.}
\label{fig}
\end{figure}

In general, the function $h^\divideontimes$ is \emph{not} a (level-set) rearrangement of $h$, since the maximum value of $h$ may vary due to $\mF$, but it has the following important properties. 
\begin{prop}\label{prop:prop}
 Let $h:\overline{\Omega}\to\R$ be a continuous radial function such that $\int_\Omega h = 0$. Then, for any $t\geq 1$,
 \begin{equation}\label{prop:eq}
   \int_\Omega h^\divideontimes = 0,\qquad \int_\Omega |h^\divideontimes|^t = \int_\Omega |h|^t,\qquad \text{ and}\qquad \cI(\mF h) = |\cI h|\geq 0 \quad \text{ in }\Omega.
\end{equation}
\end{prop}
\begin{proof}
The second claim in \eqref{prop:eq} follows by \eqref{Lp}, the definition of $f^\divideontimes$, the fact that $|\mF h|=|h|$ in $\Omega$, and polar coordinates, since
\begin{align*}
 \int_\Omega |h^\divideontimes|^t=N\omega_N\int_\delta^1 |({\mathfrak F} h)^\#(\omega_N \rho^N-\omega_N \delta^N)|^t \rho^{N-1}\ d\rho
 =\int_I |({\mathfrak F} h)^\#|^t
 =\int_\Omega |{\mathfrak F}h|^t=\int_\Omega |h|^t.
\end{align*}
For the last property in \eqref{prop:eq}, we claim that,
\begin{equation}\label{pp}
\cI(\chi_{\{\cI h>0\}}h) = \chi_{\{\cI h>0\}}\cI h\ \ \text{and}\ \ \cI(\chi_{\{\cI h\leq 0\}}h) = \chi_{\{\cI h\leq 0\}}\cI h\qquad \text{ in }\ \ [\delta,1].
\end{equation}
Indeed, by continuity and since $\cI h(\delta)=\cI h(1)=0$, the set $\{\cI h>0\}$ is open in $(\delta,1)$  and therefore 
$\{\cI h>0\} = \cup_{i=1}^\infty I_i$, where $(I_i)_{i\in\N}$ are disjoint open intervals in $(\delta,1)$. Let $I_i=(a,b)$ for some $\delta\leq a<b \leq 1$, then $N\omega_N\int_{I_i} h(s) s^{N-1}\ ds = \cI h(b)-\cI h(a) =0$, since $a,b\in\partial\{\cI h>0\}$.

Let $r>\delta$ arbitrary and let 
\begin{align*}
\zeta:=\max \{\ x :  x\in [\delta,r]\cap\{\cI h\leq 0\}\ \}\in\{\cI h\leq 0\}
\end{align*}
(recall that $\cI h(\delta)=0$ and hence $\zeta$ is well defined).  If $\zeta=r$ then $\chi_{\{\cI h>0\}}(r)=0$ and $\{\cI h>0\}\cap [\delta,r]$ is open in $(\delta,r)$, therefore   
\begin{align*}
 \cI(h \chi_{\{\cI h>0\}})(r)&=N\omega_N\int_\delta^r h(s) \chi_{\{\cI h>0\}}(s)s^{N-1}\ ds \\
& =N\omega_N\sum_{i=1}^\infty\int_{I_i\cap(\delta,r)} h(s)s^{N-1}\ ds=0=\cI h(r) \chi_{\{\cI h>0\}}(r).
\end{align*}
If $\zeta<r$, then $\chi_{\{\cI h>0\}}(r)=1$, $\cI h(\zeta)=0$, and
\begin{align*}
\cI(h \chi_{\{\cI h>0\}})(r)&=N\omega_N\int_\delta^r h(s)\chi_{\{\cI h>0\}}(s)s^{N-1}\ ds\\
&=N\omega_N\int_{\zeta}^r h(s)s^{N-1}\ ds=\cI h(r)-\cI h(\zeta)=\cI h(r)\chi_{\{\cI h>0\}}(r).
\end{align*}
Thus the first equality in \eqref{pp} follows, but then, by the linearity of the integral,
\begin{align*}
\cI(h\,\chi_{\{\cI h\leq 0\}})=\cI(h) - \cI(h\,\chi_{\{\cI h> 0\}})=\cI(h)\chi_{\{\cI h\leq 0\}}, 
\end{align*}
which shows the second equality in \eqref{pp}. The last property in \eqref{prop:eq} follows from the definition of $\mF h$, because
\begin{align*}
 \cI(\mF h) = \cI(\chi_{\{{\cal I}h>0\}}h)-\cI(\chi_{\{{\cal I}h\leq 0\}}h)
 =\cI(h)\chi_{\{{\cal I}h>0\}}-\cI(h)\chi_{\{{\cal I}h\leq 0\}}
 = |\cI h|\geq 0.
\end{align*}
Finally, since 
$\int_\Omega h^\divideontimes = \int_I (\mF h)^\# = \int_\Omega \mF h = \cI (\mF h)(1)=|\,\cI h(1)\,|=|\int_\Omega h|=0$, the first 
equality in \eqref{prop:eq} also holds, and the proof is finished.
\end{proof}

\subsection{Monotonicity results}
In this subsection we show the energy-decreasing property of the $\divideontimes$-transformation. We introduce first a useful notation.  Recall that $\Omega,$ $\delta$ are as in \eqref{Om} and $I:=[0,|\Omega|]$.  For  $h\in L_{rad}^1(\Omega)$ let
\begin{align}
&\tau:[\delta,1]\to I,\qquad \tau(r):=\omega_N (r^N-\delta^N)\nonumber\\
&\widetilde h:I\to \R,\qquad \quad\widetilde h(s):=h \circ \tau^{-1}(s)= h ((\omega_N^{-1}s+\delta^N)^\frac{1}{N})\label{hath}
\end{align}
(recall that we use $h(|x|)=h(x)$).  The function $\widetilde h$ is the one-dimensional equimeasurable version of $h$, since, for $s\in I$ and $r=\tau^{-1}(s)\in [\delta,1]$,
\begin{align}\label{equimeas}
 \int_0^{s} \widetilde h=N\omega_N\int_\delta^{r} \widetilde h(\tau(\rho)) \rho^{N-1}\ d\rho=N\omega_N\int_\delta^{r} h(\rho)\rho^{N-1}\ d\rho
 =\cI h(r).
\end{align}
Similarly, observe that 
\begin{align*}
 |\{h>t\}|=\int_\Omega \chi_{\{h>t\}}=N\omega_N\int_\delta^1 \chi_{\{h>t\}}(r) r^{N-1}\ dr=\int_0^{|\Omega|}\chi_{\{\widetilde h>t\}}=|\{\widetilde h>t\}|\quad \text{ for } t\in\R,
\end{align*}
and therefore,
\begin{align}\label{equal}
 \widetilde h^\#(s)=(h(\tau^{-1}(s)))^\#=h^{\#} (s)\qquad \text{ for }s\in I.
\end{align}
\begin{proof}[Proof of Theorem \ref{thm:mono:intro}]
Let $f$ and $g$ as in the statement.
By Proposition \ref{prop:prop}, we have that $(f^\divideontimes,g^\divideontimes)\in X$ and $\int_\Omega \frac{|f|^\alpha}{\alpha}+ \frac{|g|^\beta}{\beta} = \int_\Omega \frac{|f^\divideontimes|^\alpha}{\alpha}+ \frac{|g^\divideontimes|^\beta}{\beta},$
 so it suffices to show that $\int_\Omega g K f \leq \int_\Omega g^\divideontimes K f^\divideontimes.$  Let $(u,v):=(K g,K f)$ and observe that 
 \begin{align*}
-\frac{(u_r r^{N-1})_r}{r^{N-1}}=g\qquad \text{ and }\qquad -\frac{(v_r r^{N-1})_r}{r^{N-1}}=f\qquad \text{ in }\quad [\delta,1].
 \end{align*}
 Then, 
 \begin{align}\label{der}
u_r(r)=- \frac{\cI g(r)}{N\omega_N}r^{1-N}\qquad \text{ and }\qquad  v_r(r)=- \frac{\cI f(r)}{N\omega_N}r^{1-N}\qquad \text{ for }r\in[\delta,1].
 \end{align}
Thus, by Proposition \ref{prop:prop}, 
\begin{align}
 \int_\Omega g K f
 &=N\omega_N\int_\delta^1 -(u_r r^{N-1})_r\ v
 =N\omega_N\int_\delta^1 u_r v_r\ r^{N-1}\ dr\\
 &=(N\omega_N)^{-1}\int_\delta^1 \cI f(r)\cI g(r)\ r^{1-N}\ dr\leq (N\omega_N)^{-1}\int_\delta^1 |\cI f(r)||\cI g(r)|\ r^{1-N}\ dr\\
 &=(N\omega_N)^{-1}\int_\delta^1 \cI(\mF f)(r)\cI(\mF g)(r)\ r^{1-N}\ dr.\label{de:1}
 \end{align}
By \eqref{equimeas},
 \begin{align}
\int_\delta^1 \cI(\mF f)(r)&\cI(\mF g)(r)\ r^{1-N}\ dr
 =\int_\delta^1\int_0^{\tau(r)} \widetilde{\mF f}(\sigma)\ d\sigma \int_0^{\tau(r)} \widetilde{\mF g}(\sigma)\ d\sigma\ r^{1-N}\ dr\nonumber\\
&=(N\omega_N)^{-1}\int_0^{|\Omega|}\int_0^{{s}} \widetilde{\mF f}({t})\ d{t} \int_0^{{s}} \widetilde{\mF g}(\sigma)\ d\sigma\ (\omega_N^{-1}{s}+\delta^N)^{2(\frac{1}{N}-1)}\ d{s}.\label{de:2}
 \end{align}
 Let $\varphi({s}):= (N\omega_N)^{-1}(\omega_N^{-1}{s}+\delta^N)^{2(\frac{1}{N}-1)}\geq 0$. 
 Since $\widetilde{\mF f}^\#$ is non-increasing  in $I$ and $\int_I \widetilde{\mF f}^\#=\cI (\mF f)(1)= 0$ (by \eqref{equimeas} and because averages are preserved by rearrangements), we have that 
 $\int_0^{{s}} \widetilde{\mF f}^\#\geq 0$ for ${s}\in I$ . Moreover,  by Proposition \ref{prop:prop}, $\int_0^{{s}} \widetilde{\mF g}=\cI(\mF g)(\tau^{-1}(s))\geq 0$ for $s\in I$.  Therefore,
 \begin{align*}
 t\mapsto \int_t^{|\Omega|}  \int_0^{{s}} \widetilde{\mF g}(\sigma)\ d\sigma\,\varphi({s})\ d{s}\qquad \text{ and }\qquad \sigma\mapsto \int_\sigma^{|\Omega|}  \int_0^{{s}} \widetilde{\mF f}^\#(t)\ dt\,\varphi({s})\ d{s}
 \end{align*}
are non-increasing in $I$. Then, by (the Hardy-Littlewood inequality) Theorem \ref{thm:HL} and Fubini's theorem,
 \begin{align}
&\int_I  \int_0^{{s}} \widetilde{\mF f}({t})\ d{t} \int_0^{{s}}\widetilde{\mF g}(\sigma)\ d\sigma\,\varphi({s})\ d{s}
  =\int_I\widetilde{\mF f}({t}) \int_{t}^{|\Omega|}  \int_0^{{s}}\widetilde{\mF g}(\sigma)\ d\sigma\,\varphi({s})\ d{s}\ d{t}\nonumber\\
 &\leq \int_I\widetilde{\mF f}^\#({t}) \int_{t}^{|\Omega|}  \int_0^{{s}}\widetilde{\mF g}(\sigma)\ d\sigma\,\varphi({s})\ d{s}\ d{t}
 = \int_I \widetilde{\mF g}(\sigma)\int_\sigma^{|\Omega|}  \int_0^{{s}}\widetilde{\mF f}^\#({t})\ d{t}\,\varphi({s})\ d{s}\ d\sigma\label{in2}\\
 &\leq \int_I \widetilde{\mF g}^\#(\sigma)\int_\sigma^{|\Omega|}  \int_0^{{s}}\widetilde{\mF f}^\#({t})\ d{t}\,\varphi({s})\ d{s}\ d\sigma
 = \int_I \int_0^{{s}}\widetilde{\mF g}^\#(\sigma)\ d\sigma  \int_0^{{s}}\widetilde{\mF f}^\#({t})\ d{t}\,\varphi({s})\ d{s}.\qquad \label{in}
 \end{align}
However, using \eqref{equimeas}, \eqref{equal} and reasoning as in \eqref{de:2}, we have 
\begin{align*}
&\int_I \int_0^{{s}}\widetilde{\mF g}^\#(\sigma)\ d\sigma  \int_0^{{s}}\widetilde{\mF f}^\#({t})\ d{t}\,\varphi({s})\ d{s}
=\int_I \int_0^{s} (\mF f)^{\#} (\sigma) d \sigma \int_0^{s} (\mF g)^{\#}(\sigma)d \sigma\, \varphi(s)\, ds\\
&=\int_I \int_0^{s} \widetilde{ f^{\divideontimes}} (\sigma)\ d \sigma \int_0^{s} \widetilde{g^{\divideontimes}}(\sigma)\ d \sigma\, \varphi(s)\, ds
=\int_\delta^1 \cI(f^\divideontimes)(r)\cI(g^\divideontimes)(r)\ r^{1-N}\ dr.
\end{align*}
%}
Therefore, arguing as in \eqref{de:1}, we obtain that 
\begin{align}\label{eq:Kstar}
 \int_\Omega g K f\leq (N\omega_{N})^{-1}\int_\delta^1 \cI(f^\divideontimes)(r)&\cI(g^\divideontimes)(r)\ r^{1-N}\ dr = \int_\Omega g^\divideontimes K f^\divideontimes,
\end{align}
and \eqref{goal:mono:intro} follows. 

\medskip

We now show that if $f$ and $g$ are nontrivial and $\phi(f^\divideontimes,g^\divideontimes)=\phi(f,g)$, then $f$ and $g$ are monotone in the radial variable. Observe that, since $f$ is nontrivial, $\widetilde{\mF f}^\#$ is non-increasing, and $\int_I\widetilde{\mF f}^\#=\cI (\mF h)(1)= 0$, then $\int_0^{{s}}\widetilde{\mF f}^\#>0$ for all ${s}\in (0,|\Omega|)$. But then the function ${\sigma}\mapsto \int_{\sigma}^{|\Omega|}  \int_0^{{s}}\widetilde{\mF f}^\#({t})\ d{t}\,\varphi({s})\ d{s}$ is a strictly decreasing continuous function in $I$ and, in virtue of Lemma \ref{eq:lemma}, equality in \eqref{in} may only hold if 
$\widetilde{\mF g}=\widetilde{\mF g}^\#$ a.e. in $I$, which, by \eqref{hath}, implies that $\mF g$ coincides a.e. with a radially monotone function. Then $g$ must also be radially monotone, because $|\mF g|=|g|$ in $\Omega$, $g$ is a nontrivial continuous function, and $\int_\Omega g =0$.  Arguing similarly using \eqref{in2} we conclude that $f$ is monotone in the radial variable as well.  As a consequence, if $(u,v):=(K_{p} g,K_{q} f)$, then $u$ and $v$ are strictly monotone in the radial variable by \eqref{der} and the fact that either $\cI f>0$, $\cI g>0$ or $\cI f<0$, $\cI g<0$ in $(\delta,1)$. This ends the proof.
\end{proof}

\begin{remark}\label{rem:in}
To motivate the definition of $f^\divideontimes$, observe that the flipping $\mF f$ is very natural in virtue of \eqref{de:1}, but further insight on $\mF f$ can be gained by observing that, by \eqref{der}, the (Neumann radial) solution of $-\Delta v = \mF f$ in $\Omega$ satisfies that $v_r(r)=-(N\omega_N)^{-1}r^{1-N}\cI (\mF f)(r) \leq 0$ in $(0,1)$, i.e. $v$ is decreasing in the radial variable.  By \eqref{de:1}, this is indeed a very simple way to guarantee that there is at least one minimizer which is monotone in the radial variable; but the flipping $\mF f$ is too simple to extract any further information.  To show that \emph{all} least energy solutions $(u,v):=(K_{p} g,K_{q} f)$ are \emph{strictly} monotone decreasing, we need to
combine the flipping $\mF$ with the decreasing rearrangement $\#$ in order to apply the powerful properties of rearrangements (Theorem \ref{thm:HL} and Lemma \ref{eq:lemma}).
\end{remark}

\begin{remark}\label{ce:rem}
In general, it is \emph{not} true that $\phi(f^*,g^*)\leq \phi(f,g)$, where $*$ denotes the \emph{Schwarz symmetrization} and $(f,g)$ is as in Theorem \ref{thm:mono:intro}. A first observation is that $f^*$ is defined in $\Omega^*$ (a ball with the same measure as $\Omega$), and this would be a first complication when considering annular domains. But even if $\Omega$ is a unitary ball centered at zero $B$ and we consider the scalar problem (\emph{i.e.}, $p=q$, see Lemma \ref{lemma:p=q}), we exhibit next a continuous radial function $f:B\to\R$ with $\int_{B} f =0$ such that $\phi(f^*,f^*)>\phi(f,f)$.  Let $N=5$, $\varepsilon=\frac{1}{2}$, and $f(x):= -(\omega_N |x|^N+\varepsilon)^{-\frac{1}{2}}+k$, where $k:=2\omega_N^{-1}((\omega_N+\varepsilon)^{\frac{1}{2}}-\varepsilon^{\frac{1}{2}})$ is such that $\int_{B} f=0$. In particular, $\widetilde f(s)=-(s+\varepsilon)^{-\frac{1}{2}}+k$ for $s\in I=[0,\omega_N]=[0,\frac{8\pi^2}{15}]$. Observe that $\widetilde f$ is a strictly increasing function; therefore, by \eqref{equal}, $f^*(x)=f^\#(\omega_N\, |x|^N)=\widetilde f^\#(\omega_N\, |x|^N)=\widetilde f(\omega_N-\omega_N|x|^N)$.  We have that
\begin{align*}
\int_I \left(\int_0^s \widetilde f\right)^2(\omega_N^{-1}{s})^{2(\frac{1}{N}-1)}\ ds
\approx 5.22726 > 2.24448 \approx 
\int_I \left(\int_0^s {f^\#}\right)^2(\omega_N^{-1}{s})^{2(\frac{1}{N}-1)}\ ds.
%\int_I \left(\int_0^s \widetilde f\right)^2\varphi(s)\ ds \approx 0.198612 > 0.0852799 \approx \int_I \left(\int_0^s {f^\#}\right)^2\varphi(s)\ ds,
\end{align*}
These integrals can be computed via hypergeometric functions; for simplicity, here we just give a numerical approximation. Then, since ${f^\#}=\widetilde f^*$ in $I$, we may use the techniques and notation from the proof of Theorem~\ref{thm:mono:intro} and we find that
\begin{align*}
\int_{B} fKf=(N\omega_N)^{-1}\int_I \left(\int_0^s \widetilde f\right)^2\varphi(s)\ ds>(N\omega_N)^{-1}\int_I \left(\int_0^s \widetilde{f^*}\right)^2\varphi(s)\ ds=\int_{B} f^*Kf^*, 
\end{align*}
where $\varphi({s}):= (N\omega_N)^{-1}(\omega_N^{-1}{s})^{2(\frac{1}{N}-1)}$ for $s\in I$. Therefore $\phi(f^*,f^*)>\phi(f,f)$, because $\|f^*\|_\alpha=\|f\|_\alpha$ by \eqref{Lp}.
\end{remark}

\begin{remark}
The Neumann b.c. are used at \eqref{der} (if $\Omega$ is an annulus) and in the integration by parts performed in \eqref{de:1}. The Neumann b.c. is also the reason to consider zero-average functions. 
\end{remark}

\begin{theo}\label{coro:mono}(Sublinear case)
Let $(p,q)$ satisfy \eqref{sub} and $\Omega$ and $\delta$ be as in \eqref{Om}.  The set of minimizers of $\phi$ in $X_{rad}$ is nonempty. If $(f,g)\in X_{rad}$ is such that $\phi(f,g)=\inf_{X_{rad}}\phi$, then $f$ and $g$ are  monotone in the radial variable and, if $(u,v):=(K_{p} g,K_{q} f)$, then $u_rv_r>0$ in $(\delta,1)$; that is, $u$ and $v$ are both strictly monotone increasing or strictly monotone decreasing in the radial variable.
\end{theo}
\begin{proof}
Arguing as in Lemma \ref{minimum} we have that the set of minimizers of $\phi$ in $X_{rad}$ is nonempty and contains only nontrivial functions. By Proposition \ref{p:reg} we have that all such minimizers $(f,g)$ are continuous up to the boundary $\partial \Omega$ and $(u,v):=(K_{p} g,K_{q} f)$ solves \eqref{NHS} pointwise, in particular $(f,g)=(|u|^{p-1}u,|v|^{q-1}v)$. By Theorem \ref{thm:mono:intro}, $u$ and $v$ are strictly monotone.  To show that $u_rv_r>0$ in $(\delta,1)$, assume without loss of generality that $u$ is increasing; then $u_r>0$ in $(\delta,1)$ and, by \eqref{NHS}, \eqref{comp} and \eqref{der}, $v_r=-(N\omega_N)^{-1}\cI f(r) r^{1-N}=-(N\omega_N)^{-1}\cI (|u|^{p-1}u)(r) r^{1-N}>0$ in $(\delta,1)$, thus $v$ is also strictly increasing.  This ends the proof.
\end{proof}

\begin{theo}\label{coro:mono2}(Superlinear case)
Let $(p,q)$ satisfy \eqref{super} and $\Omega$ and $\delta$ be as in \eqref{Om}. The set of minimizers of $\phi$ in $\cN_{rad}$ is nonempty.
If $(f,g)\in {\cal N}_{rad}$ is such that $\phi(f,g)=\inf_{{\cal N}_{rad}}\phi$, then $f$ and $g$ are  monotone in the radial variable
and, if $(u,v):=(K_{p} g,K_{q} f)$, then $u_rv_r>0$ in $(\delta,1)$.
\end{theo}
\begin{proof}
Arguing precisely as in Lemma \ref{lemma:B}, there exists at least one minimizer $(f,g)$ of $\phi$ in $\cN_{rad}$. Since the $L^s$ norms are preserved by the $\divideontimes$-transformation, and \eqref{eq:Kstar} holds, then reasoning as in part 3) of the proof of Lemma \ref{lemma:B} we have that $(f^\divideontimes,g^\divideontimes)$ is also a minimizer of $\phi$ in $\mathcal{N}_{rad}$.  The rest of the proof follows as in the proof of Theorem \ref{coro:mono}.
\end{proof}

Since the proof of Theorem \ref{coro:mono} is based on one-dimensional techniques, one can also apply them to a least energy solution in the case $N=1$ to obtain monotonicity properties.

\begin{coro}\label{N1}
Let $\Omega=(-1,1)$, $(f,g)$ as in \eqref{lee}, and $(u,v):=(K_{p} g,K_{q} f)$. Then $u_rv_r>0$ in $(-1,1)$, that is, $u$ and $v$ are both strictly monotone increasing or strictly monotone decreasing in the radial variable.
\end{coro}
\begin{proof}
For $N=1$ and $\Omega=(-1,1)$, let ${\cal I}_1:L^\infty(\Omega)\to C(\overline{\Omega})$ and $\mF_1: C(\overline{\Omega})\to L^\infty( \Omega)$ be given by
\begin{align*}
{\cal I}_1h(x)=\int_{-1}^{x}h(\rho)\ d\rho \qquad  \text{ and }\qquad \mF_1 h:=(\chi_{\{{\cal I}_1h>0\}}-\chi_{\{{\cal I}_1h\leq 0\}})\, h.
\end{align*}
Let $\tau_1:[-1,1]\to[0,2]$, $x\mapsto \tau(x)=x+1$, and for $h\in C(\overline{\Omega})$, the $\divideontimes_1$-transformation of $h$ is given by 
\begin{align*}
h^{\divideontimes_1}\in L^\infty( \Omega),\qquad h^{\divideontimes_1}(x) := ({\mathfrak F_1} h)^\#(x + 1).
\end{align*}
The proof now follows by replacing ${\cal I}_1,$ $\mF_1$, $\tau_1$, and ${\divideontimes_1}$ instead of ${\cal I},$ $\mF$, $\tau$, and ${\divideontimes}$ in the proofs of Proposition \ref{prop:prop} and Theorems \ref{coro:mono} and \ref{coro:mono2} using $\delta:=-1$ and doing some obvious changes (for example, the use of polar coordinates is not needed).
\end{proof}

\section{Symmetry and symmetry breaking of minimizers}\label{sec:SSBM}

\subsection{Symmetry breaking in radial domains}\label{sec:symmetrybreaking}

This subsection is devoted to the following result.
 \begin{theo}\label{nr:tm:intro}
 Let $N\geq 1$, $\Omega$ as in \eqref{Omega:eqs}, \eqref{sc:intro} hold,
 and $(f,g)$ satisfy \eqref{lee}, then $f$ and $g$ are not radial.
 \end{theo}
The proof of this result is long, and we split it in several lemmas. The argument goes by contradiction and is also based on the results from Section \ref{SEC:monotonicity}, \emph{i.e.}, that if $f$ and $g$ are radial, then both functions are either strictly increasing or strictly decreasing. This allows to construct a direction $(\varphi,\psi)$ along which the second variation of $\phi$ at $(f,g)$ is strictly negative, contradicting the minimality property of $(f,g)$.  In the following, $\Omega\subset \R^N$ is either a ball or an annulus as in \eqref{Om}.  We begin with an auxiliary lemma. 

\begin{lemma}\label{aux:prop}
 Let $U\subset \R^N$ be a domain of class $C^{2,1}$ and
  \begin{align}\label{wu}
{\cal W}(U) :=\{ w\in C^1(\overline{U})\::\: \nabla w(x) \neq 0 \text{ whenever } x\in\overline{U} \text{ satisfies } w(x)=0 \},
 \end{align}
which is an open subset of $C^1(\overline{U})$. For $s>0$ the following holds. 
\begin{enumerate}
 \item[(1)] There is $\delta>1$ depending only on $s$ such that the map $\gamma:{\cal W}(U)\to L^\delta(U)$, $\gamma(w):=|w|^{s-1}$ is well-defined and continuous.
\item[(2)] If $w\in {\cal W}(U)$, then $(|w|^{s-1}w)_{x_i}=s|w|^{s-1}w_{x_i}$, $i=1,\ldots,N$, and
\begin{align*}
 \int_U |w|^{s-1}w \ \partial_{x_i}h = s \int_U |w|^{s-1}w_{x_i}\ h\qquad \text{ for all } h\in C^\infty_c(U).
\end{align*}
\end{enumerate}
 \end{lemma}

The proof of Lemma \ref{aux:prop} follows from (the proof of) \cite[Proposition 2.3]{PW15}, where only the case $s\in(0,1)$ is considered, but the case $s\geq 1$ follows easily from standard arguments.

The next result state that the second variation of $\phi$ 
is nonnegative along some directions.  We use $X_{rad}$ to denote the subspace of radially symmetric functions in $X$.

\begin{lemma}\label{lemma:secondderivative}
Let $pq<1$ and $(f,g)\in X_{rad}$ be a global minimizer of $\phi$ in $X$, $(u,v):=(K_{p} g , K_{q} f)$, and let $(\varphi,\psi) \in [C^1(\overline \Omega)]^2\cap X$, with $\supp(\varphi)\subset \Omega\setminus f^{-1}(0)$ and $\supp(\psi)\subset \Omega\setminus g^{-1}(0)$. Then the map $s\mapsto \phi(f+s\varphi,g+s\psi)$ belongs to $C^\infty(\R)$ and
\begin{align}\label{eq:second_derivative}
\lim_{s\to 0}\frac{d^2}{ds^2} \phi(f+s\varphi,g+s\psi)&=\int_\Omega  \frac{1}{p}|u|^{1-p}\varphi^2 + \frac{1}{q} |v|^{1-q} \psi^2 - \varphi K \psi - \psi K \varphi\, dx\geq 0.
\end{align}
\end{lemma}
\begin{proof}
For $s\in\R$, we write $\phi(f+s\varphi,g+s\psi)=\Psi(s)-T(s)$, with
\begin{align*}
\psi(s)&=\int_\Omega \frac{|f+s\varphi|^\alpha}{\alpha} +  \frac{|g+s\psi|^\beta}{\beta}\, dx \qquad \text{ and }\\
T(s)&=\int_\Omega (f+s\varphi)K(g+s\psi)= \int_\Omega fKg + s \int_\Omega \varphi Kg + fK\psi \, dx + s^2\int_\Omega \varphi K \psi.
\end{align*}
Then $T\in C^\infty(\R)$ and, by integration by parts,
\[
T''(0)=2\int_\Omega \varphi K\psi\, dx=\int_\Omega \varphi K\psi + \psi K\varphi\, dx.
\]

As for $\Psi$, since $\alpha=\frac{p+1}{p}>1$ and $\beta=\frac{q+1}{q}>1$, it is standard to show that
\[
\Psi'(s)=\int_\Omega |f+s\varphi|^{\alpha-2}(f+s\varphi) \varphi + \int_\Omega |g+s\psi|^{\beta-2}(g+s\psi) \psi.
\]
Since $u\in C^2(\overline \Omega)$ and $U:=\supp(\varphi)\subset \Omega\setminus f^{-1}(0)$, then $f=|u|^{p-1}u$ is bounded away from 0 on $U$, and $|f|^{\alpha-2}\in C^1(U)$. Analogously, $|g|^{\beta-2}\in C^1(V)$ with $V:=\supp(\psi)$.  Therefore $\psi\in C^\infty(\R)$ and
\begin{align*}
\psi''(0)=
&\lim_{s\to 0} \frac{1}{s}(\int_{U} \left(|f+s\varphi|^{\alpha-2}(f+s\varphi) - |f|^{\alpha-2}f \right) \varphi+\int_{V} \left(|g+g\varphi|^{\alpha-2}(g+s\psi) - |g|^{\beta-2}g \right) \psi)\\
=& \int_{U} (\alpha-1)|f|^{\alpha-2}\varphi^2\, dx  +  \int_{V} (\beta-1) |g|^{\beta-2}\psi^2,
\end{align*}
as required.
\end{proof}

Let $\cN_{rad}$ be the subset of radial functions in $\cN$.   We say that a function $\varphi:\overline{\Omega}\to\R$ is antisymmetric with respect to $x_1$ if $\varphi(x_1,x')=-\varphi(-x_1,x')$ for all $(x_1,x')\in\overline{\Omega}$.
\begin{lemma}\label{lemma:secondderivative:2}
Let $pq>1$ and $(f,g)\in \cN_{rad}$ be a minimizer of $\phi$ in $\cN_{rad}$, $(u,v):=(K_{p} g , K_{q} f)$.  Let $(\varphi,\psi) \in [C^1(\overline \Omega)]^2\cap X$ be antisymmetric with respect to $x_1$ and such that $\supp(\varphi)\subset \Omega\backslash f^{-1}(0)$ and $\supp(\psi)\subset \Omega\backslash g^{-1}(0)$. Then the map $s\mapsto \phi(f+s\varphi,g+s\psi)$ belongs to $C^\infty(\R)$ and \eqref{eq:second_derivative} holds. 
\end{lemma}
\begin{proof}
By Lemma \ref{lemma:secondderivative} the map $s\mapsto \phi(f+s\varphi,g+s\psi)$ is $C^2$ at $s=0$.
We now argue as in \cite[Theorem 4.2]{PW15}. Assume, by contradiction, that there is $(\varphi,\psi) \in [C^1(\overline \Omega)]^2\cap X$ antisymmetric with respect to $x_1$ satisfying $\supp(\varphi)\subset \Omega\backslash f^{-1}(0)$, $\supp(\psi)\subset \Omega\backslash g^{-1}(0)$, and such that 
$\frac{d^2}{ds^2} \phi(f+s\varphi,g+s\psi)<-2\kappa$ for $s\in(-\kappa_1,\kappa_1)$ and some $\kappa,\kappa_1>0$. 
Let $\gamma_1$ and $\gamma_2$ as in \eqref{gammas}. Using the fact that $(f,g)$ is radially symmetric and $(\varphi,\psi)$ is antisymmetric with respect to $x_1$, we obtain that  
$\phi'(c^{\gamma_1}f,c^{\gamma_2}g)(c^{\gamma_1}\varphi,c^{\gamma_2}\psi)=0$ for $c\in[\frac{1}{2},2]$.  Then, by a Taylor expansion,
\begin{align*}
 \phi(c^{\gamma_1}(f+s\varphi),c^{\gamma_1}(g+s\psi))\leq \phi(c^{\gamma_1}f,c^{\gamma_1}g)-2\kappa s^2+o(s^2)\qquad \text{ for }s\in(-\kappa_1,\kappa_1),
\end{align*}
where $o(s^2)$ is a remainder term uniform in $c\in[\frac{1}{2},2]$. For $s\in(-\kappa_1,\kappa_1)$ sufficiently small, let $t=t({f+s\varphi,g+s\psi})>0$ be given by Lemma \ref{lemma:A}; using \eqref{eq:projection}, we may assume that $t\in [\frac{1}{2},2]$, then, by Lemma \ref{lemma:A} and using that $(f,g)\in \cN$,
\begin{equation}\label{ccc}
\begin{aligned}
 \phi(t^{\gamma_1}(f+s\varphi),t^{\gamma_2}(g+s\psi))&=\sup_{c\in[\frac{1}{2},2]}\phi(c^{\gamma_1}(f+s\varphi),c^{\gamma_2}(g+s\psi)) \\
 &\leq \sup_{c\in[\frac{1}{2},2]}\phi(c^{\gamma_1}f,c^{\gamma_2}g)-\kappa s^2 = \phi(f,g)-\kappa s^2<\phi(f,g),
\end{aligned}
\end{equation}
which yields a contradiction to the minimality of $(f,g)$, because $(t^{\gamma_1}(f+s\varphi),t^{\gamma_2}(g+s\psi))\in \cN$, and the claim follows. 
\end{proof}

In the case of least energy radial solutions, Lemma \ref{lemma:secondderivative} and Lemma \ref{lemma:secondderivative:2} allow us to conclude the following. Let $\Omega(e_1):=\{x\in \Omega\::\: x_1>0\}$.

\begin{lemma}\label{yo1}
Let $(f,g)\in X_{rad}$ satisfy \eqref{leer} and 
$(u,v):=(K_{p} g , K_{q} f)$.  If $u$ is increasing in the radial variable and $(\bar f,\bar g):=( p |u|^{p-1} u_{x_1} , q |v|^{q-1} v_{x_1} ),$ then 
\begin{align}\label{ap:pos:1}
K \bar f\geq 0\ \ \text{ and }\ \ K \bar g\geq 0\quad \text{ in }\Omega(e_1)
\end{align}
 and
\begin{align}\label{st:pos}
\infty> \int_{\Omega(e_1)} p|u|^{p-1}u_{x_1}(u_{x_1}-K \bar g )+q|v|^{q-1} v_{x_1}(v_{x_1}-K \bar f)\, dx\geq 0,
\end{align}
\end{lemma}
\begin{proof}
Note that $(u,v):=(K_{p} g , K_{q} f)\in [C^{2,\varepsilon}(\overline{\Omega})]^2\ \backslash\ \{(0,0)\}$ is a radial classical solution of \eqref{NHS}, by Proposition \ref{p:reg}. 
By Theorem \ref{thm:maintheorem3} we know that $u_rv_r>0$ in $(\delta,1)$, with $\delta=\inf_{x\in\Omega}|x|$.  By assumption $u_r\geq 0$ in $(\delta,1)$, thus $u$ and $v$ are strictly monotone \emph{increasing} in the radial variable. In particular,
\begin{align*}
\bar f:=p |u|^{p-1} u_{x_1}\geq 0\qquad\text{ and }\qquad \bar g:=q |v|^{q-1} v_{x_1}\geq 0\qquad \text{ in } \Omega(e_1).
\end{align*}
Moreover, since $u$ and $v$ are sign-changing, there exist $r_1,r_2\in (\delta,1)$ such that 
\begin{equation}\label{eq:r1_and_r2}
u^{-1}(0)=f^{-1}(0)=\{x\in \Omega:\ |x|=r_1\}\quad \text{and}\quad v^{-1}(0)=g^{-1}(0)=\{x\in \Omega:\ |x|=r_2\},
\end{equation}
i.e., the nodal sets are two spheres contained in $\Omega$ and $u,v\in {\cal W}(\Omega)$, with ${\cal W}(\Omega)$ as defined in \eqref{wu}. To control these nodal lines we use the following cutoff functions: for $\varepsilon,r\geq 0$, let $\rho_\varepsilon^r$ be a smooth radial function in $\R^N$ such that
\begin{align*}
 0\leq \rho^r_\varepsilon\leq 1,\qquad \rho_\varepsilon^r(x)=0\text{ if }|\,|x|-r\,|<\varepsilon,\qquad \text{ and }\qquad \rho_\varepsilon^r(x)=1\text{ if }|\,|x|-r\,|>2\varepsilon,
\end{align*}
and let $(\bar f_\varepsilon,\bar g_\varepsilon):=(\bar f\rho^{r_1}_\varepsilon,\bar g\rho^{r_2}_\varepsilon)$ for $\varepsilon\geq 0$; note that $(\bar f_\varepsilon,\bar g_\varepsilon)\in [C^1(\overline{\Omega})]^2$ for $\varepsilon>0$ and $(\bar f_\varepsilon,\bar g_\varepsilon)$ is antisymmetric in $\Omega$ with respect to $x_1$  for all $\varepsilon\geq 0$, because $(u,v)$ is radially symmetric in $\Omega$. Moreover, 
\begin{align*}
(K \bar g_\varepsilon,K \bar f_\varepsilon)\in[C^{2}(\Omega)\cap C^{1}(\overline{\Omega})]^2\ \ \text{ for }\varepsilon>0,\qquad (K \bar g_0 , K \bar f_0)=(K \bar g , K \bar f)\in[W^{2,t}(\Omega)]^2
\end{align*}
for some $t>0$ (see Lemma \ref{aux:prop}), and $(K \bar g_\varepsilon,K \bar f_\varepsilon)$ is also antisymmetric in $\Omega$ with respect to $x_1$ for all $\varepsilon\geq 0$, by uniqueness.  In particular, $K \bar f_\varepsilon=K \bar g_\varepsilon=0$ on $\partial \Omega(e_1)\cap\{x_1=0\}$ and 
\begin{align}\label{ap:pos:2}
K \bar g_\varepsilon\geq 0\quad \text{ in }\Omega(e_1)\qquad \text{ and }\qquad K \bar f_\varepsilon\geq 0\quad \text{ in }\Omega(e_1),
\end{align}
by the maximum principle and Hopf's boundary point lemma. Since $(\bar f_\varepsilon,\bar g_\varepsilon)\to (\bar f,\bar g)$ in $L^t(\Omega(e_1))$ as $\varepsilon\to 0$, then 
\begin{align}\label{conv:pos} 
(K\bar f_\varepsilon,K\bar g_\varepsilon)\to (K\bar f,K\bar g)\quad \text{ in }W^{2,t}(\Omega(e_1))\quad \text{ as }\varepsilon\to 0, 
\end{align}
and therefore \eqref{ap:pos:2} implies \eqref{ap:pos:1}. Furthermore, since the product of two antisymmetric functions is symmetric, we have that
\begin{align}
2\int_{\Omega(e_1)} &\bar f_\varepsilon(u_{x_1}\rho^{r_1}_\varepsilon-K \bar g_\varepsilon)+\bar g_\varepsilon(v_{x_1}\rho^{r_2}_\varepsilon-K \bar f_\varepsilon)\, dx\nonumber\\
&=\int_{\Omega} \bar f_\varepsilon(u_{x_1}\rho^{r_1}_\varepsilon-K \bar g_\varepsilon)+\bar g_\varepsilon(v_{x_1}\rho^{r_2}_\varepsilon-K \bar f_\varepsilon)\, dx\geq 0,\label{eq:s}
\end{align}
by applying Lemma \ref{lemma:secondderivative} in the sublinear case \eqref{sub} or Lemma \ref{lemma:secondderivative:2} in the superlinear case \eqref{super} with  $(\varphi,\psi):=(\bar f_\eps,\bar g_\eps)$ for small $\varepsilon>0$.  The claim \eqref{st:pos} follows from Lebesgue's dominated convergence theorem once we show the existence of a suitable majorant. 
Indeed, let $(\xi_\varepsilon,\zeta_\varepsilon):=(K \bar g_\varepsilon,K \bar f_\varepsilon)$ for $\varepsilon\geq 0$, then, by \eqref{conv:pos},
\[
\bar f_\varepsilon(u_{x_1}\rho_\varepsilon^{r_1}-\xi_\varepsilon)\to |u|^{p-1}u_{x_1}(u_{x_1}-\xi_0)\qquad  \text{ a.e. in }\Omega(e_1)\quad \text{ as $\varepsilon\to 0$}.
\]
Moreover, we have that $-\Delta (\xi_\varepsilon-\xi_0)=\bar g_\varepsilon-\bar g=q|v|^{q-1}v_{x_1}(\rho_\varepsilon^{r_2}-1)\leq 0$ in $\Omega(e_1)$, with $\partial_\nu(\xi_\varepsilon-\xi_0)=0 \text{ on } \partial\Omega(e_1)\backslash \{x_1=0\}$ and $\xi_\varepsilon-\xi_0=0$ on $\partial\Omega(e_1)\cap \{x_1=0\}$. Thus (testing the equation with $(\xi_\varepsilon-\xi_0)_+$ and integrating by parts), $0<\xi_\varepsilon\leq \xi_0$ in $\Omega(e_1)$  for $\varepsilon$ small, and
\begin{align}\label{majorant}
\left|\bar f_\varepsilon(u_{x_1}\rho_\varepsilon^{r_1}-\xi_\varepsilon)\right|\leq |u|^{p-1}u^2_{x_1}  +  |u|^{p-1}u_{x_1}\xi_\varepsilon \leq |u|^{p-1}u^2_{x_1}  +  |u|^{p-1}u_{x_1}\xi_0. 
\end{align}
Note that $|u|^{p-1}u^2_{x_1}\in L^1(\Omega)$, because $|u|^{p-1}\in L^t(\Omega)\subset L^1(\Omega)$ and $u\in C^2(\overline \Omega)$.  It remains to show that
\begin{align}\label{final}
|u|^{p-1}u_{x_1}\xi_0\in L^1(\Omega) 
\end{align}
Since  $-\Delta(u_{x_1}-\xi_0)=0$ in $\Omega$ in the strong sense, interior elliptic regularity (see, \emph{e.g.}, \cite[Theorem 9.19]{GT98} ) implies that $u_{x_1}-\xi_0\in C^\infty(\Omega)$; therefore $\xi_0\in C^1(\Omega)\cap W^{2,t}(\Omega)$, because $u_{x_1}\in C^1(\overline{\Omega})$. This directly implies \eqref{final} for $p\geq 1$.  If $p\in(0,1)$, then let $\gamma>0$ be such that 
$A:=\{x\in \Omega: r_1-\gamma<|x|<r_1+\gamma\}\subset \Omega $.  Then
\begin{align*}
\int_{\Omega(e_1)} |u|^{p-1} |u_{x_1}\xi_0| &=\int_{\Omega(e_1)\setminus A} |u|^{p-1}|u_{x_1}\xi_0| + \int_A |u|^{p-1} |u_{x_1}\xi_0| \\
&\leq (\min_{\Omega(e_1)\setminus A} |u|)^{p-1} \| u_{x_1} \|_{L^\infty(\Omega)} \int_{\Omega(e_1)} |\xi_0|+\| u_{x_1} \xi_0\|_{L^\infty(A)}\int_{\Omega(e_1)}|u|^{p-1}  <\infty,
\end{align*}
and \eqref{final} also follows.  A majorant for the term $\bar g_\varepsilon(v_{x_1}\rho^{r_2}_\varepsilon-K \bar f_\varepsilon)$ in \eqref{eq:s} can be obtained similarly, and this ends the proof.
\end{proof}

The following lemma shows that an antisymmetric Neumann solution dominates the corresponding Dirichlet solution in a half radial domain $\Omega(e_1)$.

\begin{lemma}\label{yo2}
 Let $t>1$ and $h\in L^t(\Omega)\backslash\{0\}$ be an antisymmetric function in $\Omega$ with respect to $x_1$ and let $w^N:=K h\in W^{2,t}(\Omega)$, that is, $w^N$ is the unique strong solution of 
 \begin{align*}
 -\Delta w^N=h\quad \text{ in }\Omega,\qquad\partial_\nu w^N=0\quad \text{ on }\partial \Omega,\qquad \text{ and }\qquad \int_\Omega w^N =0.
 \end{align*}
Moreover, let $w^D\in W^{2,t}(\Omega)\cap C^1(\overline{\Omega})$ be a strong solution of $-\Delta w^D=h$ in $\Omega$ with $w^D=0$ on $\partial \Omega$.  If
$h\geq0$ and $w^N\geq 0$ in $\Omega(e_1)$ then $w^D<w^N$ in $\Omega(e_1)$.
\end{lemma}
\begin{proof}
By uniqueness of solutions we have that $w^D\in W^{2,t}(\Omega)\cap C^1(\overline{\Omega})$ and $w^N:=K h\in W^{2,t}(\Omega)$ are antisymmetric functions in $\Omega$ with respect to $x_1$ and therefore, since $w^N\geq 0$ in $\Omega(e_1)$ by assumption,
\begin{align}\label{bcss}
 w^D-w^N=0\ \ \text{ on }\partial \Omega(e_1)\cap\{x_1=0\}\quad \text{ and }\quad w^D-w^N\leq 0\ \ \text{ on }\partial \Omega(e_1)\cap\{x_1>0\}.
\end{align}
Observe that $w^D\not\equiv w^N$ in $\Omega(e_1)$, because otherwise $w^D=w^N\geq 0$ in $\Omega(e_1)$ with $\partial_\nu w^D=w^D=0$ on $\partial \Omega(e_1)\cap\{x_1>0\}$, which contradicts Hopf's boundary point lemma.  Then, since $-\Delta(w^D-w^N)=0$ in $\Omega$, $w^D\not\equiv w^N$ in $\Omega(e_1)$, and \eqref{bcss} holds, the maximum principle yields that $w^D<w^N$ in $\Omega(e_1)$.
\end{proof}

\begin{proof}[Proof of Theorem \ref{nr:tm:intro}]
Let $\Omega\subset \R^N$ as in \eqref{Om}, $(f,g)$ as in the statement, 
$(u,v):=(K_{p} g , K_{q} f)\in[C^{2,\varepsilon}(\overline{B})]^2$,
and $(\bar f,\bar g):=(p |u|^{p-1} u_{x_1},q |v|^{q-1} v_{x_1})$.  We argue by contradiction. Assume without loss of generality that $f$ is radial.  Using the relations $u=|f|^{\frac{1}{p}}f,$ $v=K_{q} f$, and $g=|v|^{q-1}v,$ we obtain that also $u$, $v$, and $g$ must be radially symmetric. By Lemma \ref{aux:prop}, there is some $t>1$ such that $|u|^{p-1},|v|^{q-1}\in L^t(\Omega)$ and $\partial_{x_1}(|v|^{q-1}v)=q|v|^{q-1}v_{x_1}\in L^t(\Omega)$. Therefore, we have that $-\Delta u = |v|^{q-1}v\in W^{1,t}(\Omega)$ and (interior) elliptic regularity (see for instance \cite[Theorem 9.19]{GT98}) yields that $u\in W_{\rm loc}^{3,t}(\Omega)$. Thus, we may interchange derivatives, and $(-\Delta u)_{x_1}=-\Delta (u_{x_1})$ in $\Omega$. Arguing analogously for $-\Delta v$ we conclude that $(u_{x_1},v_{x_1})\in [W^{2,t}(\Omega)\cap C^{1,\varepsilon}(\overline{\Omega})]^2$ is the unique strong solution of the Dirichlet problem
\begin{align}
-\Delta u_{x_1} &= q |v|^{q-1} v_{x_1},\quad -\Delta v_{x_1} = p |u|^{p-1} u_{x_1} \quad  \text{ in }\Omega\qquad\text{with}\quad u_{x_1}=v_{x_1}=0\quad \text{ on }\partial\Omega,\label{D}
\end{align}
where the boundary conditions follow from the fact that $u$ and $v$ are radially symmetric and $\partial_\nu u=\partial_\nu v=0$ on $\partial\Omega$. 
By Theorem \ref{thm:maintheorem3}, we may assume that $u$ and $v$ are strictly increasing in the radial variable (the other case follows similarly). Then $u_{x_1}$ and $v_{x_1}$ are nonnegative in $\Omega(e_1)$ and,  by Lemmas \ref{yo1} and \ref{yo2},
\begin{align*}
0&\leq \int_{\Omega(e_1)} p|u|^{p-1}u_{x_1}(u_{x_1}-K\bar g)+q|v|^{q-1} v_{x_1}(v_{x_1}-K\bar f)\, dx <0,
 \end{align*}
a contradiction.
\end{proof}

\subsection{Foliated Schwarz symmetry}

In this section we show that least energy solutions are foliated Schwarz symmetric whenever the domain $\Omega$ is a ball or an annulus centered at zero in $\R^N$ in dimension $N\geq 2$. For $N=1$, see Corollary \ref{N1}.  We introduce first some notation.  Let $\Sn=\{x\in\mathbb R^N: |x|=1\}$ be the unit sphere and fix $e\in \Sn$. We consider the halfspace $H(e):=\{x\in \mathbb R^N: x\cdot e>0\}$ and the half domain $\Omega(e):=\{x\in \Omega: x\cdot e>0\}.$  

The composition of a function $w:\overline{\Omega}\to\R$ with a reflection with respect to $\partial H(e)$ is denoted by $w_e$, that is,
\begin{align*}
 w_e: \overline{\Omega}\to\R\qquad \text{ is given by }\qquad w_e(x):=w(x-2(x\cdot e)e).
\end{align*}
 
The \emph{polarization} $u^H$ of $u:\overline{\Omega}\to \R$ with respect to a hyperplane $H=H(e)$ is given by
\begin{align*}
 u^{H}:=\begin{cases} 
      \max\{u,u_e\} & \text{ in }\overline{\Omega(e)},\\
      \min\{u,u_e\} & \text{ in }\overline{\Omega}\ \backslash \, \overline{\Omega(e)}.
 \end{cases}
\end{align*}

Following \cite{SW03}, we say that $u\in C(\overline{\Omega})$ is \textit{foliated Schwarz symmetric with respect to some unit vector $p\in \Sn$} if $u$ is axially symmetric with respect to the axis $\mathbb R p$ and nonincreasing in the polar angle $\theta:= \operatorname{arccos}(\frac{x}{|x|}\cdot p)\in [0,\pi].$   We use the following characterization of foliated Schwarz symmetry given in \cite{SW12}. We remark that these kind of characterizations appeared for the first time in \cite{B03}, see also \cite{W10} for a survey on symmetry via reflection methods.
\begin{lemma}[Particular case of Proposition 3.2 in \cite{SW12}]\label{l:char}
There is $p\in\Sn$ such that $u,v\in C(\overline{\Omega})$ are foliated Schwarz symmetric with respect to $p$ if and only if for every $e\in \Sn$ either 
 \begin{align}\label{char}
u\geq u_e, \ v\geq v_e\quad \text{ in }\Omega(e)\qquad \text{ or }\qquad u\leq u_e,\ v\leq v_e\quad \text{ in }\Omega(e).
 \end{align}
\end{lemma}

The main result of this section is the following. For similar results under Dirichlet boundary conditions, we refer to \cite[Theorem 1.3]{BMR13} and \cite[Theorem 1.2]{BMRT15}.

\begin{theo}[Sublinear case]\label{th:sym}
 Let $\Omega$ be either a ball or an annulus centered at the origin of $\R^N$, $N\geq 2$. Let $(p,q)$ satisfy \eqref{sub} and $(f,g)\in X$ be a global minimizer of $\phi$ in $X$ and $(u,v):=(K_{p} g,K_{q} f)$. There is $p\in\Sn$ such that $u$ and $v$ are foliated Schwarz symmetric with respect to $p$.
\end{theo}
\begin{proof}
Let $(f,g)$ and $(u,v)$ as in the statement and fix a hyperplane $H=H(e)$ for some $|e|=1$.  By Proposition \ref{p:reg}, $(u,v)\in [C^{2,\mu}(\Omega)]^2$ for some $\mu\in(0,1)$, $(u,v)$ solves \eqref{NHS} pointwise, and $(f,g)=(-\Delta v,-\Delta u)\in [C^{\mu}(\overline{\Omega})]^2$; thus $(f^H,g^H)\in [L^\infty(\Omega)]^2$ and $(\widetilde u,\widetilde v):=(K_{p} (g^H),K_{q} (f^H))\in [W^{2,N}({\Omega})\cap C^1(\overline{\Omega})]^2$, by Sobolev embeddings.

Let $V:=v_e+v-\widetilde v - \widetilde v_e$, then using the definition of $f^H$ we have that 
\begin{align*}
-\Delta V =f-f^H-(f^H)_e+f_e=0\quad \text{ in }\Omega, \qquad \partial_\nu V=0\quad \text{ on } \partial \Omega. 
\end{align*}
testing this equation with $V$ and integrating by parts we obtain that
 $V=k$ for some $k\in\R$. Then 
\begin{align}\label{reflections}
v_e+v=\widetilde v_e+\widetilde v+k\qquad \text{ in }\Omega
\end{align}
Let 
\begin{align}\label{Gammas}
 \Gamma_1:=\{x\in\partial{\Omega(e)}\::\: x\cdot e=0\},\qquad \Gamma_2:=\{x\in\partial{\Omega(e)}\::\: x\cdot e>0\},
\end{align}
$w_1:=\widetilde v-v+k/2$, and $w_2:=\widetilde v-v_e+k/2$. Since $v=v^e$ and $\widetilde v=\widetilde v_e$ on $\Gamma_1$ we have that 
$w_1=w_2=0$ on $\Gamma_1$, by \eqref{reflections}, and $\partial_\nu w_1=\partial_\nu w_2=0$ on $\Gamma_2$. Furthermore,
\begin{align*}
-\Delta w_1 =f^H - f\geq 0\quad \text{ in }\Omega(e)\qquad \text{ and }\qquad -\Delta w_2 =f^H - f_e\geq 0\quad \text{ in }\Omega(e),
\end{align*}
which implies by the maximum principle and Hopf's Lemma that $w_1 \geq 0$ and $w_2 \geq  0 $ in $\Omega(e)$. Therefore, using that $\widetilde v_e=v_e+v-\widetilde v-k$ (by \eqref{reflections}) and $g^H_e=g_e+g- g^H$ (by definition of $g^H$),
\begin{align}
\int_\Omega g K f- g^H K f^H\ dx&=\int_\Omega gv - g^H \widetilde v
 =\int_{\Omega(e)} gv + g_ev_e - g^H \widetilde v - (g^H)_e \widetilde v_e\, dx \nonumber\\
 &=\int_{\Omega(e)} gv + g_ev_e - g^H \widetilde v - (g_e+g- g^H)(v_e+v-\widetilde v-k)\, dx\nonumber\\
 &=\int_{\Omega(e)} (g_e-g^H)w_1 + (g-g^H)w_2 + \frac{k}{2}(g_e+g)\, dx \leq 0,\label{ineq}
\end{align}
because $\int_\Omega g = 0$.  

To show that $u$ and $v$ are foliated Schwarz symmetric with respect to the same vector $p\in\Sn$ we use Lemma \ref{l:char} and argue by contradiction. Assume that \eqref{char} does not hold. Then, without loss of generality, there are $e\in \Sn$ and the corresponding halfspace $H=H(e)$ such that
$v\neq v^H$ in $B(e)$ and either $v_e\neq v^H$ in $\Omega(e)$ or $u_e\neq u^H$ in $\Omega(e)$. Since $t\mapsto h(t):=|t|^s t$ is a strictly monotone increasing function in $\R$ for $s>-1$, this implies that
\begin{align}\label{nopol:2}
f = h(v)\neq h(v)^H = f^H  \qquad \text{ and either }\qquad f_e\neq f^H\quad \text{ or }\quad g_e\neq g^H\quad\text{ in }\Omega(e).
\end{align}
As a consequence, $w_1>0$ and $0\neq g-g^H\leq 0$ in $B(e)$ and either $w_2>0$ in $B(e)$ or $0\neq g_e-g^H\leq 0$ in $B(e)$. In any case, \eqref{ineq} implies that 
\begin{align}\label{eq:T(fH)}
\int_\Omega f K g < \int_\Omega f^H K g^H.
\end{align}  But then, since $L^s$ norms are preserved under polarizations, we obtain that $\phi(f,g)>\phi(f^H,g^H)$, a contradiction to the minimality of $(f,g)$.  Therefore \eqref{char} holds and the theorem follows from Lemma \ref{l:char}.
\end{proof}

\begin{theo}[Superlinear case]\label{th:sym2}
 Let $\Omega$ be either a ball or an annulus centered at the origin of $\R^N$, $N\geq 2$. Let $(p,q)$ satisfy \eqref{super} and $(f,g)\in X$ be a minimizer of $\phi$ in $\mathcal{N}$ and $(u,v):=(K_{p} g,K_{q} f)$. There is $p\in\Sn$ such that $u$ and $v$ are foliated Schwarz symmetric with respect to $p$.
\end{theo}
\begin{proof} Arguing by contradiction as in the proof of Theorem \ref{th:sym}, we obtain \eqref{eq:T(fH)}. Since $(f,g)\in \mathcal{N}$ and $L^s$-norms are preserved under polarizations we have, by \eqref{page13},
\[
\int_\Omega \gamma_1 |f^H|^\alpha + \gamma_2 |g^H|^\beta\, dx<\int_\Omega f^H Kg^H
\]
and $\int_\Omega f^H K g^H >0$, then by Lemma \ref{lemma:A} there exists $0<t<1$ such that $(t^{\gamma_1} f^H,t^{\gamma_2} g^H)\in \mathcal{N}$. This gives a contradiction, because, by \eqref{page13},
\begin{align}
\inf_\mathcal{N} \phi &\leq \phi(t^{\gamma_1} f^H, t^{\gamma_2} g^H)= t^\gamma (1-\gamma) \int_\Omega \frac{|f^H|^\alpha}{\alpha}+\frac{|g^H|^\beta}{\beta}\, dx\\
				& < (1-\gamma) \int_\Omega \frac{|f|^\alpha}{\alpha}+\frac{|g|^\beta}{\beta}\, dx =\phi(f,g).\qedhere
\end{align}
\end{proof}

\section{Further results}\label{FR:sec}

\subsection{Unique continuation principle for minimizers}

In this section, $\Omega$ is again a general smooth bounded domain. We prove that if $(u,v)$ is a solution of \eqref{NHS} associated to a minimizer of $\phi$, then the nodal sets $u^{-1}(0):=\{x\in \Omega:\ u(0)=0\}$ and $v^{-1}(0):=\{x\in \Omega:\ v(x)=0\}$ have zero Lebesgue measure. We do it by extending the results and techniques from \cite[Section 3]{PW15} to the setting of Hamilitonian elliptic systems and to the dual method framework.  Recall that if $(f,g)\in X$ is a critical point of $\phi$ then $(u,v):=(K_pg,K_q f)\in [C^{2,\varepsilon}(\overline \Omega)]^2$ is a classical solution of \eqref{NHS}. Our main result is the following.

\begin{theo}\label{uc:thm}
 Let $(f,g)\in X$ be a critical point of $\phi$ in $X$ and $(u,v):=(K_{p} g , K_{q} f)$,  then $u^{-1}(0)=v^{-1}(0) $ a.e.  Moreover, if $p$ and $q$ satisfy the sublinear condition  \eqref{sub} and $(f,g)\in X$ is a global minimizer of $\phi$ in $X$, then $|u^{-1}(0)|=|v^{-1}(0)|=0$.
\end{theo}

To show Theorem \ref{uc:thm} we rely on the following preliminary results.

\begin{lemma}[Lemma 3.1 in \cite{PW15}]\label{PW:lemma}
 Let $\gamma>0$ and $f:(0,\infty)\to [0,\infty)$ be such that $f$ is bounded in  $[\varepsilon,\infty)$ for every $\varepsilon>0$ and $\lim_{r\to 0}r^\gamma f(r) = 0$. Then for every $r>0$ there is $s>0$ such that $f(s)\geq f(r)$ and $f(t)\leq 2^\gamma f(s)$ for all $t\in [\frac{s}{2},2s]$.
\end{lemma}

We now characterize the decay of any solution $(u,v)$ of \eqref{NHS} close to a common zero. 
\begin{prop}\label{p:decay}
 Let $(u,v)\in [C^{2}(\overline{\Omega})]^2$ be a solution of \eqref{NHS} with $p,q>0$, $pq<1$. If $x_0\in\Omega$ is a point of density one for the set $u^{-1}(0)\cap v^{-1}(0)$ then $|u(x)|^{p+1}+|v(x)|^{q+1}=o(|x-x_0|^{\gamma})$, where
 \begin{equation}\label{eq:gamma}
 \gamma=\frac{2(p+1)(q+1)}{1-pq}=2\left(1-\frac{1}{\alpha}-\frac{1}{\beta} \right)^{-1}.
 \end{equation}
In other words,
\begin{align*}
\text{ if } \qquad \lim_{r\to 0} \frac{|u^{-1}(0)\cap v^{-1}(0)\cap B_r(x_0)| }{|B_r(x_0)|}=1 \qquad \text{ then }\qquad  \lim_{x\to x_0} \frac{|u(x)|^{p+1}+|v(x)|^{q+1}}{|x-x_0|^{\gamma}}=0.
\end{align*}
\end{prop}
\begin{proof}
Without loss of generality we assume that $x_0=0$ and we set $u\equiv v\equiv 0$ in $\R^N\backslash \Omega$. Let 
\begin{align*}
f:(0,\infty)\to [0,\infty)\qquad \text{ be given by }\qquad f(r):=r^{-\gamma}\sup_{|x|=r}(|u(x)|^{p+1}+|v(x)|^{q+1}).
\end{align*}

\noindent \emph{Step 1.} We show first that $f\in L^\infty(0,\infty)$. Indeed, assume by contradiction that there is a sequence $(r_n)_{n\in\N}$ such that $f(r_n)\to \infty$ as $n\to\infty$. Then, by Lemma \ref{PW:lemma}, there is $(s_n)_{n\in\N}$ such that $f(s_n)\geq f(r_n)$ and $f(t)\leq 2^{\gamma}f(s_n)$ for all $t\in[\frac{s_n}{2},2s_n]$ and $n\in\N$.  Then $f(s_n)\to \infty$ as $n\to\infty$ and $s_n\to 0$, by the definition of $f$.  Working if necessary with a subsequence, we may assume that $B_{2s_n}\subset \Omega$ for all $n\in\N$. Set $\Omega_0:= B_2\backslash B_{\frac{1}{2}}$ and let $u_n,v_n:\Omega_0\to \R$ be given by
\begin{align}\label{e:1}
u_n(x)=\frac{u(s_nx)}{f(s_n)^\frac{1}{p+1}s_n^{\frac{\gamma}{p+1}}}\qquad \text{ and }\qquad v_n(x)=\frac{v(s_nx)}{f(s_n)^\frac{1}{q+1}s_n^{\frac{\gamma}{q+1}}}.
\end{align}
By \eqref{NHS}, and since $2-\frac{\gamma}{p+1}+q\frac{\gamma}{q+1}=2-\frac{\gamma}{q+1}+p\frac{\gamma}{p+1}=0$  we have that 
\begin{equation}\label{e:2}
-\Delta u_n = A_n|v_n|^{q-1}v_n\quad \text{ and }\quad -\Delta v_n= A_n|u_n|^{p-1}u_n\qquad \text{  in } \Omega_0,
\end{equation}
where
\begin{equation}\label{e:3}
A_n=f(s_n)^{-\frac{2}{\gamma}}=f(s_n)^\frac{pq-1}{(p+1)(q+1)}\to 0\qquad  \text{ as } n\to \infty.
\end{equation}
Observe next that
\begin{equation}\label{e:4}
|u_n(x)|^{p+1}+|v_n(x)|^{q+1} = |x|^{\gamma}\frac{|u(s_n x)|^{p+1}+|v(s_n x)|^{q+1}}{(|x|s_n)^{\gamma}f(s_n)}\leq 2^{\gamma}\frac{f(s_n |x|)}{f(s_n)} \leq 4^{\gamma}
\end{equation}
for all $x\in \Omega_0$ and $n\in \N.$ 
By \eqref{e:2}, \eqref{e:3}, \eqref{e:4}, and interior elliptic regularity, there are subsequences $u_n\to u^*$, $v_n\to v^*$ in $C^1_{loc}(\Omega_0)$. Furthermore, by definition of $f$ and the regularity of $u,v$, there is $(x_n)_{n\in\N}\subset \mathbb S^{N-1}$ such that $|u_n(x_n)|^{p+1}+|v_n(x_n)|^{q+1}=1$ for all $n\in\N$. Thus, up to a subsequence, 
 $x_n\to x^*\in \mathbb S^{N-1}$ with 
\begin{align}\label{e:5}
|u^*(x^*)|^{p+1}+|v^*(x^*)|^{q+1}=1.
\end{align}

 However, since $0$ is a point of density one for $u^{-1}(0)\cap v^{-1}(0)$, we have that
\begin{align}\label{e:6}
 \frac{|\{x\in \Omega_0 \::\: u_n(x)\neq 0\}|}{|B_2|}\leq \frac{|\{x\in B_{2s_n} \::\: u(x)\neq 0\}|}{|B_{2s_n}|}\to 0\quad \text{ as } n\to\infty,
\end{align}
which implies that $u^*\equiv 0$ in $\Omega_0$. Analogously, we obtain that $v^*\equiv 0$ in $\Omega_0$. This contradicts \eqref{e:5} and therefore $f\in L^\infty(0,\infty)$.  

\smallbreak

\noindent \textit{Step 2.} Now, it suffices to show that $\lim_{r\to 0}f(r)=0$. We argue again by contradiction: assume there is a sequence $r_n\to 0$ as $n\to\infty$ such that $f(r_n)\geq \varepsilon$ for all $n\in\N$ and for some $\varepsilon>0$. Passing if necessary to a subsequence we have that $B_{2r_n}\subset \Omega$.  Since $f$ is bounded (by Step 1), the rescaled functions $\widetilde u_n, \widetilde v_n:\Omega_0\to \R$ given by 
\begin{equation}\label{eq:tilde_un}
\widetilde u_n(x)=\frac{u(r_nx)}{r_n^{\frac{\gamma}{p+1}}} \qquad \text{ and  } \qquad \widetilde v_n(x)=\frac{v(r_nx)}{r_n^{\frac{\gamma}{q+1}}}
\end{equation}
are uniformly bounded. Moreover, 
\begin{equation}\label{eq:system_tilde_un}
-\Delta \widetilde u_n = |\widetilde v_n|^{q-1}\widetilde v_n\qquad \text{ and }\qquad -\Delta \widetilde v_n =|\widetilde u_n|^{p-1}\widetilde u_n\qquad \text{ in }\quad  \Omega_0,
\end{equation}
by \eqref{NHS}, and there is $(\widetilde x_n)_{n\in\N}\subset \mathbb S^{N-1}$ such that $|\widetilde u_n(x_n)|^{p+1}+|\widetilde v_n(x_n)|^{q+1}=f(r_n)\geq \varepsilon$ for all $n\in\N$. But arguing as before, we obtain subsequences $\widetilde u_n \to \widetilde u$, $\widetilde v_n \to \widetilde v$ in $C^1_{loc}(\Omega_0)$, and $\widetilde x_n\to\widetilde x\in \mathbb S^{N-1}$ such that $|\widetilde u(\widetilde x)|+|\widetilde v(\widetilde x)|\geq \varepsilon$. But \eqref{e:3} with $\widetilde u_n$ instead of $u_n$ yields that $\widetilde u\equiv 0$, and we have analogously that $\widetilde v\equiv 0$, a contradiction. Therefore $\lim_{r\to 0}f(r)=0$, and the proof is finished.
\end{proof}

We now use Proposition \ref{p:decay} to construct directions along which the energy $\phi$ decreases.
\begin{prop}\label{p:pd}
 Let $(f,g)\in X$ be a critical point of $\phi$, $(u,v):=(K_{p} g , K_{q} f)$, and assume that \eqref{sub} holds. If $|u^{-1}(0)\cap v^{-1}(0)|>0$, then there is $(\varphi,\psi)\in X$ such that $\phi(f,g)>\phi(f+\varphi,g+\psi)$.
\end{prop}
\begin{proof}
If $|u^{-1}(0)\cap v^{-1}(0)|>0$ then, by Lebesgue's density theorem (see, for example, \cite[page 45]{EG92}), there exists a point of density one $x_0$ of $u^{-1}(0)\cap v^{-1}(0)$.  Without loss of generality we assume that $x_0=0$.  Since $(u,v):=(K_{p} g , K_{q} f)$ is a classical solution of \eqref{NHS} by Proposition \ref{p:reg}, we obtain by Proposition \ref{p:decay} that
\begin{align}
 |g|^{\beta}&=|\Delta u |^\beta = |v|^{q+1} =o(|x|^{\gamma})\qquad \text{ and }\qquad |f|^\alpha=|\Delta v|^\alpha = |u|^{p+1} =o(|x|^{\gamma}) \label{e:decay:2},
\end{align}
as $|x|\to 0$, where $\gamma$ is as in \eqref{eq:gamma}.  Let $\zeta\in C_c^\infty(\R^N)\backslash \{0\}$ such that $\operatorname{supp}\zeta\subset \Omega$, $\int_{\Omega}\zeta =0$, and fix $t>0$ such that
\begin{align}
c:=t^{\frac{\alpha\beta}{\alpha+\beta}} \int_{\Omega}\frac{|\zeta|^{\alpha}}{\alpha}+\frac{|\zeta|^{\beta}}{\beta}\ dx -t \frac{\|\zeta\|_{L^2(\Omega)}^4}{\|\nabla \zeta\|^2_{L^2(\Omega)}}<0\label{e:c:def}
\end{align}
(such number $t$ exists, since $\alpha\beta/(\alpha+\beta)>1$, which is equivalent to $pq<1$).

For $r>0$ small denote $\Omega_r:=r\Omega\subset \Omega$ and let $\varphi_r, \psi_r:\Omega\to \R$ be given by 
\begin{align*}
\varphi_r(x):= r^{\frac{\gamma}{\alpha}}t^{\frac{\beta}{\alpha+\beta}}\zeta(\,\frac{x}{r}\,),\quad \psi_r(x):= r^{\frac{\gamma}{\beta}}t^{\frac{\alpha}{\alpha+\beta}}\zeta(\,\frac{x}{r}\,)
\end{align*}
Notice that $\varphi_r\equiv \psi_r\equiv 0$ in $\R^N\backslash \Omega_r$ and note that $w:=K (\zeta(\frac{x}{r}))\in C^{2}(\overline{\Omega})\cap C^1(\overline{\Omega})$ solves classically $-\Delta w = \zeta(\frac{x}{r})$ in $\Omega$ and $\partial_{\nu} w =0$ on $\partial\Omega$, therefore
\begin{align}\label{e:f:c}
 \int_{\Omega}\zeta(\frac{x}{r}) K (\zeta(\frac{x}{r}))=\int_\Omega (-\Delta w)w=\int_\Omega |\nabla w|^2=\|\nabla w\|_{L^{2}(\Omega)}^2.
\end{align}
On the other hand, multiplying $-\Delta w = \zeta(\frac{x}{r})$ in $\Omega$ by $\zeta(\frac{x}{r})$ and using H\"{o}lder's inequality,
\begin{align*}
r^{N}\|\zeta\|_{L^2(\Omega)}^2&=\int_{\Omega_r} |\zeta(\frac{x}{r})|^2=\int_{\Omega} (-\Delta w) \zeta(\frac{x}{r})=\int_{\Omega} \nabla w \nabla(\zeta(\frac{x}{r}))
=r^{-1}\int_{\Omega_r} \nabla w \nabla\zeta(\frac{x}{r})\\
&\leq r^{-1}\|\nabla w\|_{L^2(\Omega)} (\int_{\Omega_r} |\nabla\zeta(\frac{x}{r})|^2)^{\frac{1}{2}}
= r^{-1}\|\nabla w\|_{L^2(\Omega)} r^\frac{N}{2}\|\nabla \zeta\|_{L^2(\Omega)},
\end{align*}
that is, $\|\nabla w\|_{L^2(\Omega)}\geq r^{1+\frac{N}{2}}\frac{\|\zeta\|_{L^2(\Omega)}^2}{\|\nabla \zeta\|_{L^2(\Omega)}},$ and therefore, by \eqref{e:f:c},
\begin{align}\label{e:l:e}
 \int_{\Omega}\zeta(\frac{x}{r}) K (\zeta(\frac{x}{r}))=\|\nabla w\|^2_{L^2(\Omega)}
 \geq r^{2+N}\frac{\|\zeta\|_{L^2(\Omega)}^4}{\|\nabla \zeta\|^2_{L^2(\Omega)}}.
\end{align}
Then, by \eqref{e:c:def},  \eqref{e:l:e}, and observing that $\gamma$ satisfies $\gamma(1/\alpha+1/\beta)+2=\gamma$ (cf. \eqref{eq:gamma}) we obtain
\begin{align}
 &\phi(\varphi_r,\psi_r)=\int_{\Omega_r}\frac{|\varphi_r|^\alpha}{\alpha}+\frac{|\psi_r|^\beta}{\beta}-\psi_r K \varphi_r \, dx\nonumber\\
 &=\int_{\Omega_r}r^{\gamma}t^{\frac{\alpha\beta}{\alpha+\beta}}\frac{|\zeta(\frac{x}{r})|^\alpha}{\alpha}+r^{\gamma}t^{\frac{\alpha\beta}{\alpha+\beta}}\frac{|\zeta(\frac{x}{r})|^\beta}{\beta}
 -r^{\gamma\left(\frac{1}{\alpha}+\frac{1}{\beta}\right)}t\zeta(\frac{x}{r}) K (\zeta(\frac{x}{r})) \, dx\nonumber\\
 &\leq r^{\gamma+N}\int_{\Omega}t^{\frac{\alpha\beta}{\alpha+\beta}}\frac{|\zeta(y)|^\alpha}{\alpha}+t^{\frac{\alpha\beta}{\alpha+\beta}}\frac{|\zeta(y)|^\beta}{\beta}\ dy
 -r^{\gamma\left(\frac{1}{\alpha}+\frac{1}{\beta}\right)}
 r^{2+N}\ t\frac{\|\zeta\|_{L^2(\Omega)}^4}{\|\nabla \zeta\|^2_{L^2(\Omega)}} \nonumber\\
 &= r^{\gamma+N}c<0\label{e:phi:r}.
\end{align}

We claim that $E:=\phi(f+\varphi_r,g+\psi_r)-\phi(f,g)<0$. Indeed, by \eqref{e:decay:2},
\begin{align}
\int_{{\Omega_r}}|f&+\varphi_r|^\alpha-|f|^\alpha-|\varphi_r|^\alpha \, dx
\leq \int_{{\Omega_r}}|o(r^{\frac{\gamma}{\alpha}})+r^{\frac{\gamma}{\alpha}}t^{\frac{\beta}{\alpha+\beta}}\zeta(\frac{x}{r})|^\alpha-r^{\gamma}|t^{\frac{\beta}{\alpha+\beta}}\zeta(\frac{x}{r})|^\alpha+o(r^{\gamma}) \, dx \nonumber\\
&\qquad \qquad \leq 
r^{\gamma}\int_{{\Omega_r}}|o(1)
+t^{\frac{\beta}{\alpha+\beta}}\zeta(\frac{x}{r})|^\alpha-|t^{\frac{\beta}{\alpha+\beta}}\zeta(\frac{x}{r})|^\alpha
+o(1)\ dx\nonumber\\
&\leq 
r^{\gamma+N}\int_{{\Omega}}|o(1)
+t^{\frac{\beta}{\alpha+\beta}}\zeta(y)|^\alpha-|t^{\frac{\beta}{\alpha+\beta}}\zeta(y)|^\alpha
+o(1)\ dy=o(r^{\gamma+N})\label{e:est:1}
\end{align}
and analogously, $\int_{{\Omega_r}}|g+\psi_r|^\beta-|g|^\beta-|\psi_r|^\beta\ dx=o(r^{\gamma+N})$ as $r\to 0$.  Furthermore, by Proposition~\ref{p:decay}, 
\begin{align}
 \int_{{\Omega_r}}|\psi_r K f|
 &= \int_{{\Omega_r}}|\psi_r v|
 \leq  \int_{{\Omega_r}}r^{\frac{q \gamma}{q+1}}|t^{\frac{\alpha}{\alpha+\beta}}\zeta(\frac{x}{r})| o(r^{\frac{\gamma}{q+1}}) \nonumber\\
 &=  o(r^{\gamma+N}) \int_{{\Omega}}|t^{\frac{\alpha}{\alpha+\beta}}\zeta|  = o(r^{\gamma+N})\label{e:est:2}
\end{align}
 as $r\to 0$, and analogously, $\int_{{\Omega_r}}|\varphi_r K g| = o(r^{\gamma+N})$ as $r\to 0$.  Therefore, by \eqref{ibyp}, \eqref{e:phi:r}, \eqref{e:est:1}, \eqref{e:est:2}, and the fact that $\varphi_r\equiv 0$ in $\R^N\backslash \Omega_r$,
\begin{align*}
 E&=\int_\Omega \frac{|f+\varphi_r|^\alpha-|f|^\alpha}{\alpha}+\frac{|g+\psi_r|^\beta-|g|^\beta}{\beta}-(g+\psi_r)K(f+\varphi_r)+gK f \, dx\\
 &=\phi(\varphi_r,\psi_r)+\int_{{\Omega_r}}\frac{|f+\varphi_r|^\alpha-|f|^\alpha-|\varphi_r|^\alpha}{\alpha}+\frac{|g+\psi_r|^\beta-|g|^\beta-|\psi_r|^\beta}{\beta}
 -\psi_r K f - \varphi_r K g \, dx \\
 &=\phi(\varphi_r,\psi_r)+o(r^{\gamma+N})
 =r^{\gamma+N}c+o(r^{\gamma+N})<0 
  \end{align*}
for $r>0$ sufficiently small, and the proof is finished.
\end{proof}

We are now ready to conclude the proof of the main result of this section.
\begin{proof}[Proof of Theorem \ref{uc:thm}]
 Let $(f,g)\in X$ be a critical point of $\phi$ in $X$ and $(u,v):=(K_{p} g , K_{q} f)$. By \cite[Lemma 7.7]{GT98}, Proposition \ref{p:reg}, and \eqref{NHS}, we have that $|v|^q=|\Delta u|=0$ a.e. in $u^{-1}(0)$ and $|u|^p=|\Delta v|=0$ a.e. in $v^{-1}(0)$, therefore $u^{-1}(0)=v^{-1}(0)$ a.e.    Now, assume that \eqref{sub} holds and $(f,g)\in X$ is a global minimizer of $\phi$ in $X$. Then, by Proposition \ref{p:pd}, we have that $|u^{-1}(0)\cap v^{-1}(0)|=0$ and thus $|u^{-1}(0)|=|v^{-1}(0)|=0$, since $u^{-1}(0)=v^{-1}(0)$ a.e.
\end{proof}

\begin{remark}(Superlinear case) \label{rem:superlinear_UCP}Let $(u,v)\in [C^{2}(\overline{\Omega})]^2$ be a solution of \eqref{NHS} with $p,q>0$, $pq>1$. If $|u^{-1}(0)\cap v^{-1}(0)|>0$, then $(u,v)$ has a \emph{zero of infinite order}; indeed, if $x_0\in\Omega$ is a point of density one for the set $u^{-1}(0)\cap v^{-1}(0)$ and $\gamma>0$, then $|u(x)|^{p+1}+|v(x)|^{q+1}=o(|x-x_0|^{\gamma})$ as $x\to x_0$.  A proof can be obtained by repeating almost word by word the proof of Proposition \ref{p:decay} with the following changes: in Step 1, the functions $u_n,v_n$ defined in \eqref{e:1} satisfy system \eqref{e:2} with $\widetilde A_n:= s_n^2 (\sup_{|x|=s_n} |u(x)|^{p+1}+|v(x)|^{q+1})^\frac{pq-1}{(p+1)(q+1)}$ instead of $A_n$ as in \eqref{e:3}, observe that $\widetilde A_n\to 0$ as $n\to \infty$; moreover, in Step 2, the functions $\widetilde u_n,\widetilde v_n$ defined in \eqref{eq:tilde_un} solve (instead of \eqref{eq:system_tilde_un}) the system $-\Delta \widetilde u_n=B_n |\widetilde v_n|^{q-1} \widetilde v_n, -\Delta \widetilde v_n=B_n |\widetilde u_n|^{p-1}\widetilde u_n$, with $B_n=r_n^{2+\gamma\frac{pq-1}{(p+1)(q+1)}}\to 0$ as $n\to\infty$.
\end{remark}

\subsection{Simplicity of the zeros for radial solution}

 \begin{proof}[Proof of Theorem \ref{aux:lemma:intro}]
  Let
  \begin{align}\label{t:def}
  t(r)=r^{2-N}\quad \text{ if }N\geq 3\qquad \text{ or }\qquad t(r)=1- \log (r)\quad \text{ if }N=2,   
  \end{align}
  $b:=\lim_{s\to \delta}t(s)\in(1,\infty]$ ($b=\infty$ if $\Omega=B_1$), $J:=[1,b)$, and $\varphi,\psi:J\to\R$ satisfy $\varphi (t(|x|))=u(x)$ and $\psi (t(|x|))=v(x)$.  Then $(\varphi,\psi)$ is a bounded solution of 
  \begin{align}\label{t:eqs}
   -\varphi''=\zeta\ |\psi|^{q-1}\psi,\quad -\psi''=\zeta\ |\varphi|^{p-1}\varphi \quad \text{ in } J\quad \text{ with }\quad  \varphi'=\psi'=0\quad \text{ on }\partial J,
  \end{align}
 where $\zeta(t):=(N-2)^{-2} t^{-2\frac{N-1}{N-2}}\geq 0$ for $N\geq 3$ and $\zeta(t)=e^{-2(t-1)}\geq 0$ for $N=2$. 
 
 \medbreak
 
 To prove \eqref{Weps} it suffices to show that 
 \begin{align}\label{claim:l}
  |\varphi|+|\varphi'|>0\quad \text{ and }\quad |\psi|+|\psi'|>0\qquad \text{ in }J.
 \end{align} 
We argue by contradiction: assume there is $x\in J$ such that 
\begin{align}\label{cont:hip}
\psi(x)=\psi'(x)=0 
\end{align}
Then, by \eqref{t:eqs}, 
\begin{align}\label{cont:l}
 \text{there is }x_0\in(x,b)\text{ such that }\varphi(x_0)=0.
\end{align}
Indeed, if $|\varphi|>0$ in $(x,b)$ then $|\psi''|>0$ in $(x,b)$ and, in virtue of \eqref{cont:hip}, $\psi$ must be unbounded if $b=\infty$ or $|\psi'(b)|>0$ if $b<\infty$, which is impossible by assumption, and therefore \eqref{cont:l} follows. 

Now, it suffices to rule out the next three cases: 
\begin{align*}
a)\quad \varphi(x)=0,\qquad 
b)\quad\varphi(x)\varphi'(x)\geq 0,\qquad c)\quad \varphi(x)\varphi'(x)< 0. 
\end{align*}
   \smallbreak
   \noindent \emph{Case a)} 
Assume that $\varphi(x)=0$. Take the first integral $\Phi:J\to\R$ given by 
\[
\Phi:=\varphi'\psi'+\zeta\ (\frac{|\varphi|^{p+1}}{p+1}+\frac{|\psi|^{q+1}}{q+1}),
\]
which satisfies
 \begin{align}\label{d}
  \Phi'=\zeta'\ (\frac{|\varphi|^{p+1}}{p+1}+\frac{|\psi|^{q+1}}{q+1})\leq 0\qquad \text{ in } J,
  \end{align}
by \eqref{t:eqs} and the fact that $\zeta$ is decreasing. Moreover, since $\varphi(x)=0$, \eqref{cont:hip} implies that $\Phi(x)=0$ and therefore
 \begin{align}\label{ent}
 \varphi'\psi'\leq -\zeta(\frac{|\varphi|^{p+1}}{p+1}+\frac{|\psi|^{q+1}}{q+1})\leq 0 \quad \text{ in } \ [x,b),
  \end{align}
  by \eqref{d}. By \eqref{cont:l}, $\varphi(x_0)=0$ for some $x_0>x$.  We claim that $\varphi\equiv 0$ in $[x,x_0]$. By \eqref{t:eqs}, this implies that also $\psi\equiv 0$ in $[x,x_0]$, which contradicts the assumption that $\varphi^{-1}(0)\cap \psi^{-1}(0)$ has empty interior, ruling out case a).
  
  To prove the claim, suppose that $\varphi\not\equiv 0$, and take  $y^*\in (x,x_0)$ such that $|\varphi(y^*)|\neq 0$. Without loss of generality, assume that $\varphi(y^*)>0$ (the other case is analogous). Then \eqref{ent} implies $ \varphi'(y^*)\psi'(y^*)<0$ and (by replacing the point $y^*$ if necessary) we may assume that $\varphi'(y^*)>0$. But then, using \eqref{ent} and a simple continuity argument, we have $\varphi,\varphi'>0$ in $[y^*,b)$, in contradiction with  \eqref{cont:l}.
  \smallbreak
  
\noindent \emph{Case b)}  Let $\varphi(x)>0$ and $\varphi'(x)\geq 0$ (the case $\varphi(x)<0$ and $\varphi'(x)\leq 0$ is similar). 
By adjusting the point $x_0$ from \eqref{cont:l} if necessary, we may assume that $\varphi>0$ in $(x,x_0)$, and by system \eqref{t:eqs} also $-\psi''>0$ in $(x,x_0)$.  By \eqref{cont:hip}, this implies that $\psi<0$ in $(x,x_0)$ and therefore $-\varphi''<0$ in $(x,x_0)$, \emph{i.e.}, $\varphi$ is convex in $(x,x_0)$, which would contradict $\varphi(x_0)=0$.
  \smallbreak
  
 \noindent \emph{Case c)} Finally, we rule out $\varphi(x)\varphi'(x)<0$.  Suppose that $\varphi(x)>0$ and $\varphi'(x)<0$ (the case $\varphi(x)<0$ and $\varphi'(x)>0$ is similar). Observe that, by  \eqref{t:eqs} we deduce $\psi''(x)<0$, which, combined with \eqref{cont:hip}, yields $\psi<0$ in some interval $(w_0,x)\subseteq (1,x)$. We claim that actually $\psi<0$ in $(1,x)$, which yields a contradiction since this would imply $\varphi''>0$ in $(1,x)$ and thus $\varphi'(1)<0$, contradicting the boundary conditions in \eqref{t:eqs} . 
 
 To prove the claim, assume $w_0>1$ and $\psi(w_0)=0$.  Then, by the mean value theorem, there is $w_1\in(w_0,x)$ such that $\psi'(w_1)=0$.  Since $\psi'(x)=0$ we may use again the mean value theorem to obtain $w_2\in(w_1,x)$ such that $\psi''(w_2)=0$. But then $\varphi(w_2)=0$ by \eqref{t:eqs}, which is impossible since $\varphi(x)<0$, $\varphi'(x)>0$, and $-\varphi''=\zeta |\psi|^{q-1}\psi>0$ in $(w_0,x)$. 
  
  \smallbreak
  
  Since we have reached a contradiction in all cases, we conclude that \eqref{cont:hip} cannot happen. The case $\varphi(x)=\varphi'(x)=0$ is completely analogous to \eqref{cont:hip} and therefore the claim \eqref{claim:l} follows.

 \end{proof}
 
\textbf{Acknowledgments.} 
A. Salda\~{n}a is supported by the Alexander von Humboldt Foundation, Germany.  H. Tavares is partially supported by ERC Advanced Grant 2013 n. 339958 ``Complex Patterns for Strongly Interacting Dynamical Systems - COMPAT''. Both authors were partially supported by FCT/Portugal through UID/MAT/04459/2013.  The graphs in Figure 1 and the numerical approximations in Remark \ref{ce:rem} were done using Mathematica 11.1, Wolfram Research Inc. The authors would like to thank Tobias Weth for fruitful discussions and the anonymous referee for the careful reading of the paper and helpful comments and suggestions.

%\bibliographystyle{plain}
%\bibliography{sl}

\end{document}